\documentclass[a4paper, 11pt]{article}

\usepackage{fullpage}

\usepackage[english]{babel}
\usepackage[utf8]{inputenc}
\usepackage{amsmath, amsfonts, amssymb, amsthm}
\usepackage{mathtools}
\usepackage{bm}
\usepackage{authblk}

\usepackage[colorlinks=true, allcolors=blue]{hyperref}
\usepackage{booktabs}
\usepackage{enumitem}
\usepackage{graphicx}
\usepackage{subcaption}

\usepackage[linesnumbered,ruled,vlined]{algorithm2e}
\SetKwInput{KwInput}{input}
\SetAlFnt{\small}

\usepackage[backend=biber, style=nature]{biblatex}
\newcommand{\Real}{\mathbb{R}}
\newcommand{\N}{\mathbb{N}}

\newcommand{\E}{\mathbb{E}}
\newcommand{\Prob}{\mathbb{P}}
\newcommand{\Id}{\mathrm{Id}}

\newtheorem{theorem}{Theorem}[section]
\newtheorem{lemma}[theorem]{Lemma}

\newtheorem{definition}{Definition}[section]

\title{Tensor train completion: local recovery guarantees via Riemannian optimization}
\author[1]{Stanislav Budzinskiy}
\author[1]{Nikolai Zamarashkin}
\affil[1]{Marchuk Institute of Numerical Mathematics RAS}
\date{}

\begin{document}
\maketitle

\begin{abstract}
In this work, we estimate the number of randomly selected elements of a tensor that with high probability guarantees local convergence of Riemannian gradient descent for tensor train completion. We derive a new bound for the orthogonal projections onto the tangent spaces based on the harmonic mean of the unfoldings' singular values and introduce a notion of core coherence for tensor trains. We also extend the results to tensor train completion with auxiliary subspace information and obtain the corresponding local convergence guarantees.
\end{abstract}

\tableofcontents


\section{Introduction}

The problem of recovering algebraically structured data from scarce measurements has already become a classic one. The data under consideration are typically sparse vectors or low-rank matrices and tensors, while the measurements are obtained by applying a linear operator that satisfies a variant of the so-called \textit{restricted isometry property (RIP)} \cite{CandesTaoDecoding2005}. 

In this work, we focus on tensor completion, which consists in recovering a tensor in the tensor train (TT) format \cite{OseledetsTyrtyshnikovBreaking2009, OseledetsTensorTrain2011} from a small subset of its entries. Specifically, we consider it as a Riemannian optimization problem \cite{AbsilEtAlOptimization2009, UschmajewVandereyckenGeometric2020} on the smooth manifold of tensors with fixed TT ranks and derive sufficient conditions (essentially, the RIP) for local convergence of the Riemannian gradient descent. We further estimate the number of randomly selected entries of a tensor with low TT ranks that is sufficient for the RIP to hold with high probability and, as a consequence, for the Riemannian gradient descent to converge locally. We leave aside the the question of producing a starting point that lies close enough to the solution and concentrate instead on reducing the required number of samples.

Before presenting our main contributions to tensor completion, we suggest to step back and take a look at a particular case of two-dimensional tensors with low TT ranks, which is low-rank matrices. The research into matrix completion, as opposed to multi-dimensional tensor completion, is more mature: not only are the properties of the matrix completion problem better understood, but also the key ideas that were produced in the process have been extended to the tensor case and greatly influenced its development. Therefore, starting with matrix completion is both historically motivated and allows one to better grasp the main concepts: the notion of \textit{coherence} and the RIP. They lie at the heart of low-rank matrix completion, allowing one to prove that there are computationally feasible methods to solve the problem.


Let $A \in \Real^{n_1 \times n_2}$ be a rank-$r$ matrix and let $\Omega \subseteq [n_1] \times [n_2]$ with $[k] = \{ 1, \ldots, k\}$ be a collection of indices. Assuming that $A(i_1,i_2)$ are known for $(i_1,i_2) \in \Omega$, we aim to find a matrix $X \in \Real^{n_1 \times n_2}$ that solves the following rank minimization problem:
\begin{equation}
\label{intro:eq:rank_min}
    \text{rank}(X) \to \min \quad \text{s.t.} \quad X(i_1,i_2) = A(i_1,i_2), \, (i_1,i_2) \in \Omega.
\end{equation}
Two important questions arise: what are the requirements for \eqref{intro:eq:rank_min} to have a unique solution and whether the problem is computationally tractable.

Rank minimization problems with affine constraints are NP-hard in general \cite{RechtEtAlGuaranteed2010}, and Fazel \cite{FazelMatrix2002} developed a heuristic that consists in minimizing the nuclear norm, i.e. the sum of the singular values
\begin{equation*}
    \| X \|_* = \sum\nolimits_{k = 1}^{\min(n_1, n_2)} \sigma_k(X).
\end{equation*}
The matrix completion problem \eqref{intro:eq:rank_min} then turns into a convex optimization problem
\begin{equation}
\label{intro:eq:nuclear_min}
    \| X \|_* \to \min \quad \text{s.t} \quad X(i_1,i_2) = A(i_1,i_2), \, (i_1,i_2) \in \Omega,
\end{equation}
and can be solved as a semidefinite program. A breakthrough in understanding the properties of the nuclear norm minimization for matrix completion was achieved by Cand\`es, Recht, and Tao \cite{CandesRechtExact2009a, CandesTaoPower2010, RechtSimpler2011} who established sufficient conditions under which $A$ is the unique solution to \eqref{intro:eq:nuclear_min}. Their main contribution consists in showing that these sufficient conditions hold with high probability provided that sufficiently many indices $\Omega$ are chosen uniformly at random. To this end, the authors introduced several key notions and assumptions that limit the class of matrices amenable for completion. 

The \textit{coherence of a linear subspace} is one of them. For an $r$-dimensional linear subspace $T$ of $\Real^{n}$, its coherence $\mu(T)$ is defined as
\begin{equation}
\label{intro:eq:coherence}
    \mu(T) = \frac{n}{r} \max_{i \in [n]} \| \mathcal{P}_T e_i \|^2_2, \quad 1 \leq \mu(T) \leq \frac{n}{r},
\end{equation}
where $e_i \in \Real^n$ are canonical basis vectors and $\mathcal{P}_T : \Real^n \to T$ is the orthogonal projection operator. With a slight abuse of notation, we will write $\mu(U) = \mu(T)$ for any matrix $U$ whose columns span $T$. If $U$ happens to have orthonormal columns, the coherence of its column space can be computed as
\begin{equation*}
    \mu(U) = \frac{n}{r} \max_{i \in [n]} \| U^T e_i \|^2_2.
\end{equation*}
The worst case for matrix completion is a rank-$1$ matrix of the form $A = e_i e_j^T$: there is no hope for recovery unless we observe all of its entries. Similarly pessimistic are $A = u e_j^T$ and $A = e_i v^T$. For these examples, their column and/or row spaces have the maximum possible coherences. A reasonable assumption, then, is that both column and row spaces of $A$ are incoherent, i.e. their coherences are bounded by a small constant
\begin{equation}
\label{intro:eq:incoherent}
    \mu(U) \leq \mu_0, \quad \mu(V) \leq \mu_0.
\end{equation}
Here, $U \in \Real^{n_1 \times r}$ and $V \in \Real^{n_2 \times r}$ are the left and right singular factors of $A$.

Another object that plays an important role in \cite{CandesRechtExact2009a, CandesTaoPower2010, RechtSimpler2011} is the following linear subspace of $\Real^{n_1 \times n_2}$, associated with $A$,
\begin{equation}
\label{intro:eq:tangent_space}
    T_A = \{ UM + NV^T ~:~ M \in \Real^{r \times n_2},\, N \in \Real^{n_1 \times r} \} \subset \Real^{n_1 \times n_2}
\end{equation}
together with the corresponding orthogonal projection operator 
\begin{equation}
\label{intro:eq:tangent_proj}
    \mathcal{P}_{T_A} X = U U^T X + X V V^T - U U^T X V V^T \in T_A.
\end{equation}
In fact, $T_A$ is exactly the \textit{tangent space at $A$ to the smooth manifold of rank-$r$ matrices} \cite{LeeIntroduction2012}. Let $\mathcal{R}_{\Omega} : \Real^{n_1 \times n_2} \to \Real^{n_1 \times n_2}$ denote the sampling operator that sets to zero all elements of a matrix that do not lie in the index set $\Omega$:
\begin{equation}
\label{intro:eq:sampling_operator}
    \mathcal{R}_{\Omega} X = \sum\nolimits_{(i_1,i_2) \in \Omega} X(i_1,i_2) e_{i_1} e_{i_2}^T.
\end{equation}
One of the assumptions made in \cite{CandesRechtExact2009a, CandesTaoPower2010, RechtSimpler2011} to prove that $A$ is the unique solution to \eqref{intro:eq:nuclear_min} is that $\mathcal{R}_{\Omega}$ satisfies a variant of the RIP with $\varepsilon = 1/2$:
\begin{equation}
\label{intro:eq:rip}
    \| \rho^{-1} \mathcal{P}_{T_A} \mathcal{R}_{\Omega} \mathcal{P}_{T_A} - \mathcal{P}_{T_A} \| < \varepsilon, \quad \rho = \tfrac{|\Omega|}{n_1 n_2},
\end{equation}
where $\| \cdot \|$ is the operator norm induced by the Frobenius norm. Calling \eqref{intro:eq:rip} a RIP is justified, since its direct consequence is
\begin{equation*}
    (1 - \varepsilon) \| X \|_F \leq \| \rho^{-1} \mathcal{P}_{T_A} \mathcal{R}_{\Omega} X \|_F \leq  (1 + \varepsilon) \| X \|_F, \quad X \in T_A.
\end{equation*}
Cand\`es and Recht proved estimates on the number of known elements $|\Omega|$ that guarantees the RIP \eqref{intro:eq:rip}.

\begin{theorem}[\cite{CandesRechtExact2009a}, Theorem 4.1 and \cite{RechtSimpler2011}, Theorem 6]
\label{intro:theorem:recht_rip}
Let the matrix $A$ have incoherent column and row spaces \eqref{intro:eq:incoherent} and assume that the index set $\Omega$ is chosen uniformly at random with
\begin{equation*}
    |\Omega| \gtrsim \frac{1}{\varepsilon^2} \mu_0 r n \log(n), \quad n = \max(n_1, n_2).
\end{equation*}
Then the RIP \eqref{intro:eq:rip} holds with high probability.
\end{theorem}

The RIP \eqref{intro:eq:rip} alone, however, is not sufficient for $A$ to be the unique minimizer of \eqref{intro:eq:nuclear_min}. The best (to date) estimate on the number of known elements that guarantees that nuclear norm minimization \eqref{intro:eq:nuclear_min} solves the matrix completion problem was derived in \cite{DingChenLeave2020} with the help of leave-one-out analysis.

\begin{theorem}[\cite{DingChenLeave2020}, Theorem 2]
\label{intro:theorem:nnm_recovery}
Let the matrix $A$ have incoherent column and row spaces \eqref{intro:eq:incoherent} and assume that the index set $\Omega$ is chosen uniformly at random with
\begin{equation*}
    |\Omega| \gtrsim \mu_0 r \log(\mu_0 r) n \log(n), \quad n = \max(n_1, n_2).
\end{equation*}
Then $A$ is the unique minimizer of \eqref{intro:eq:nuclear_min}.
\end{theorem}
This bound is almost optimal, considering that
\begin{equation*}
    |\Omega| \gtrsim \mu_0 r n \log(n)
\end{equation*}
random samples are necessary to rule out the situation when several rank-$r$ matrices agree on the sample $\Omega$ (see \cite{CandesTaoPower2010}). It is also interesting to note that both necessary and sufficient conditions amount to only polylogarithmic oversampling as $r(n_1 + n_2 - r)$ parameters describe every rank-$r$ matrix of size $n_1 \times n_2$.

A different approach to matrix completion is to minimize the residual on the sampling set under the rank constraint:
\begin{equation}
\label{intro:eq:nonconvex}
    \| \mathcal{R}_{\Omega}X - \mathcal{R}_{\Omega}A \|^2_F \to \min \quad \text{s.t.} \quad \text{rank}(X) \leq r.
\end{equation}
Unlike \eqref{intro:eq:nuclear_min}, this optimization problem is non-convex and, as a result, can have multiple local minima and saddle points. A closely related perspective builds upon a geometric fact that the set
\begin{equation*}
    \mathcal{M}_r = \{ X \in \Real^{n_1 \times n_2} ~:~ \text{rank}(X) = r\}
\end{equation*}
is a smooth embedded submanifold of $\Real^{n_1 \times n_2}$ (see \cite{LeeIntroduction2012, UschmajewVandereyckenGeometric2020}). This means that the problem
\begin{equation}
\label{intro:eq:riemann}
    \| \mathcal{R}_{\Omega}X - \mathcal{R}_{\Omega}A \|^2_F \to \min \quad \text{s.t.} \quad X \in \mathcal{M}_r
\end{equation}
can be solved using Riemannian optimization methods \cite{VandereyckenLowRank2013}. The Riemannian gradient descent (RGD) reads as
\begin{equation}
\label{intro:eq:rgd}
    X_{t + 1} = \text{SVD}_r \left(X_t - \alpha_t \mathcal{P}_{T_{X_t} \mathcal{M}_r} \left[ \mathcal{R}_{\Omega} X_t - \mathcal{R}_{\Omega} A \right]  \right),
\end{equation}
where $\alpha_t > 0$ is the step size, $T_{X_t} \mathcal{M}_r$ is the tangent space to $\mathcal{M}_r$ at $X_t \in \mathcal{M}_r$ given by \eqref{intro:eq:tangent_space}, and $\mathcal{P}_{T_{X_t} \mathcal{M}_r}$ is the corresponding orthogonal projection operator \eqref{intro:eq:tangent_proj}. The local linear convergence of the RGD \eqref{intro:eq:rgd} was studied in \cite{WeiEtAlGuarantees2016}, where it was proved that the RIP \eqref{intro:eq:rip} is, essentially, the only sufficient condition.

\begin{theorem}[\cite{WeiEtAlGuarantees2016}, Theorem 2.2]
\label{intro:theorem:wei_riemann}
Assume that the sampling operator $\mathcal{R}_\Omega$ is bounded $\| \mathcal{R}_{\Omega} \| \leq C$ and satisfies the RIP \eqref{intro:eq:rip} with $\varepsilon < 1/22$. If the initial point $X_0 \in \mathcal{M}_r$ satisfies
\begin{equation*}
    \frac{\| X_0 - A \|_F}{\sigma_{\min}(A)} < \frac{\varepsilon \sqrt{\rho}}{2 C (1 + \varepsilon)},
\end{equation*}
where $\sigma_{\min}(A)$ is the smallest positive singular value of $A$, then the RGD \eqref{intro:eq:rgd} converges linearly to $A$ as
\begin{equation*}
    \| X_t - A \|_F < \left( \frac{18 \varepsilon}{1 - 4\varepsilon} \right)^t \| X_0 - A \|_F.
\end{equation*}
\end{theorem}

Our intention with this paper is to look at the notion of coherence \eqref{intro:eq:coherence} and the RIP \eqref{intro:eq:rip} in the multi-dimensional setting of tensors with low TT ranks, explore how they affect the properties of the RGD for TT completion, and relate them to estimate the required number of known elements. In pursuing our goal, we will follow a sequence of steps:
\begin{enumerate}
    \item prove a new theorem on local linear convergence of the RGD for TT completion (an analog of Theorem \ref{intro:theorem:wei_riemann});
    \item introduce a new notion of core coherence for tensors in the TT format (an extension of the coherence \eqref{intro:eq:coherence});
    \item formulate a new incoherence assumption (an analog of \eqref{intro:eq:incoherent}) and derive a new estimate on the number of randomly selected elements of a tensor in the TT format that guarantees the RIP with high probability (an analog of Theorem \ref{intro:theorem:recht_rip}).
\end{enumerate}
We set up the notation and list the basic facts about the TT format in Section~\ref{section:notation}. The next Section~\ref{section:contrib} collects what we consider to be the main contributions of our paper: their formulations, the motivation behind them, and a detailed comparison with the existing literature. Section~\ref{section:proofs} is entirely devoted to the proofs of our main results. In Section \ref{section:discussion}, we attempt to evaluate our findings and outline directions for the future research. The paper also has two Appendices: Appendix~\ref{section:related}, where we provide a broader context of tensor completion and overview other important developments in the field, and Appendix~\ref{section:side}, where we adapt our results on TT completion to a modified problem with auxiliary subspace information.
\section{Notation and preliminaries}\label{section:notation}
We denote matrices by capital letters $X, Y, Z$ and tensors by bold capital letters $\bm{X}, \bm{Y}, \bm{Z}$. An element of a $d$-dimensional tensor $\bm{X}$ at position $(i_1, \ldots, i_d)$ is marked as $\bm{X}(i_1, \ldots, i_d)$. The identity matrix of size $n$ is written as $I_n$. We denote its columns, the canonical basis vectors of $\Real^n$, by $e_j$ for all $j \in [n] = \{ 1, \ldots, n \}$, and the size of $e_j$ will be clear from the context. Calligraphic letters such as $\mathcal{P}, \mathcal{R}, \mathcal{S}$ denote linear operators acting on matrices or tensors, $\mathrm{Id}$ is the identity operator. The Frobenius norm of a matrix or tensor is denoted by $\| \cdot \|_F$. This is a Euclidean norm with the standard inner product
\begin{equation*}
    \| \bm{X} \|_F = \sqrt{\langle \bm{X}, \bm{X} \rangle_F}, \quad \langle \bm{X}, \bm{Y} \rangle_F = \sum_{i_1 = 1}^{n_1} \ldots \sum_{i_d = 1}^{n_d} \bm{X}(i_1, \ldots, i_d) \bm{Y}(i_1, \ldots, i_d).
\end{equation*}
The operator norm induced by it is marked as $\| \cdot \|$. We write $\| \cdot \|_2$ for the $l_2$ norm of a vector and the spectral norm of a matrix. 

The Kronecker product is denoted by $\otimes$, and $\circ$ stands for the outer product. For instance, for every multi-index $\omega =~(i_1, \ldots, i_d) \in~[n_1] \times \ldots \times [n_d]$, the corresponding canonical basis tensor $\bm{E}_\omega$ and its vectorization $e_{\omega}$ can be represented as
\begin{equation*}
    \bm{E}_\omega = e_{i_1} \circ \ldots \circ e_{i_d}, \quad e_{\omega} = e_{i_d} \otimes \ldots \otimes e_{i_1}.
\end{equation*}
A mode-$k$ product of a tensor $\bm{X} \in \Real^{n_1 \times \ldots \times n_d}$ with a matrix $U \in \Real^{m_k \times n_k}$ is denoted by $\times_k$ so that
\begin{equation*}
    \bm{Y} = \bm{X} \times_k U \in \Real^{n_1 \times \ldots \times n_{k-1} \times m_k \times n_{k+1} \times \ldots n_d}, \quad \bm{Y}(i_1, \ldots, i_{k-1}, j_k, i_{k+1}, \ldots, i_d) = \sum_{i_k = 1}^{n_k} \bm{X}(i_1, \ldots, i_d) U(j_k, i_k).
\end{equation*}

For a tensor $\bm{X} \in \Real^{n_1 \times \ldots \times n_d}$, its mode-$k$ flattening is a matrix of size $n_k \times \prod_{j \neq k} n_j$ denoted by $X_{(k)}$, the columns of $X_{(k)}$ are called mode-$k$ fibers. The $k$-th unfolding of $\bm{X}$ is a matrix of size $(n_1 \ldots n_k) \times (n_{k+1} \ldots n_d)$ denoted by $X^{\langle k \rangle}$. A tensor is said to be in the tensor train (TT) format \cite{OseledetsTyrtyshnikovBreaking2009, OseledetsTensorTrain2011} if each of its elements can be evaluated according to
\begin{equation*}
    \bm{X}(i_1, \ldots, i_d) = \sum_{\alpha_1 = 1}^{r_1} \ldots \sum_{\alpha_{d-1} = 1}^{r_{d-1}} G_1(i_1, \alpha_1) \bm{G}_2(\alpha_1, i_2, \alpha_2) \ldots \bm{G}_{d-1}(\alpha_{d-2}, i_{d-1}, \alpha_{d-1}) G_d(\alpha_{d-1}, i_d).
\end{equation*}
The matrices $G_1 \in \Real^{n_1 \times r_1}$, $G_d \in \Real^{r_{d-1} \times n_d}$ and the 3-dimensional tensors $\bm{G}_k \in \Real^{r_{k-1} \times n_k \times r_k}$ are called TT cores. The upper limits of the summations, $r_k \in \N$, are conventionally combined into a tuple $\bm{r} = (r_1, \ldots, r_{d-1})$ that is called the TT rank of the decomposition. To make the notation more consistent, we will write $\bm{G}_1 \in \Real^{r_0 \times n_1 \times r_1}$ and $\bm{G}_d \in \Real^{r_{d-1} \times n_d \times r_d}$ with $r_0 = r_d = 1$ for the first and last TT cores. We will also denote by $\bm{X} = [\bm{G}_1, \bm{G}_2, \ldots, \bm{G}_d]$ the TT representation itself.

Every tensor $\bm{X}$ can be represented in the TT format. This can be achieved with the TT-SVD algorithm \cite{OseledetsTensorTrain2011}, and the TT ranks of the resulting representation are equal to the ranks of the unfolding matrices $X^{\langle k \rangle}$. The unfolding matrices can be factorized as products of interface matrices $X^{\langle k \rangle} = X_{\leq k} X_{\geq k+1}^T$, which can be defined recursively as
\begin{equation}\label{eq:tt:interface_rec}
\begin{split}
    X_{\leq 1} &= G_1, \quad X_{\leq k} = (I_{n_k} \otimes X_{\leq k-1})G_{k}^L \in \Real^{(n_1 \ldots n_k) \times r_k},\\
    X_{\geq d} &= G_d^T, \quad X_{\geq k+1} = (X_{\geq k+2} \otimes I_{n_{k+1}}) (G_{k+1}^R)^T \in \Real^{(n_{k+1} \ldots n_d) \times r_k}.
\end{split}
\end{equation}
The matrices $G_k^{L} \in \Real^{r_{k-1}n_k \times r_k}$ and $G_k^R \in \Real^{r_{k-1} \times n_k r_k}$ are the left and right unfoldings of the $k$-th TT core $\bm{G}_k$, respectively.

While a tensor can admit various TT represenations with different TT ranks, under certain minimality conditions of the representation (satisfied by what TT-SVD outputs) the TT ranks are unique \cite{HoltzEtAlmanifolds2012}. Namely, for every TT core its left and right unfoldings must be full-rank. This justifies the notion of the TT rank of a tensor
\begin{equation*}
    \text{rank}_{TT}(\bm{X}) = (\text{rank}(X^{\langle 1 \rangle}), \ldots, \text{rank}(X^{\langle d-1 \rangle})).
\end{equation*}
The set of tensors of tensors with fixed TT rank will be denoted by
\begin{equation*}
    \mathcal{M}_{\bm{r}} = \{ \bm{X} \in \Real^{n_1 \times \ldots \times n_d}~:~\text{rank}_{TT}(\bm{X}) = \bm{r} \},
\end{equation*}
and it is a smooth embedded submanifold\cite{HoltzEtAlmanifolds2012, UschmajewVandereyckenGeometric2020} of $\Real^{n_1 \times \ldots \times n_d}$ of dimension 
\begin{equation*}
    \dim \mathcal{M}_{\bm{r}} = \sum_{k=1}^{d} r_{k-1} n_k r_{k} - \sum_{k=1}^{d-1} r_k^2.
\end{equation*}

Among all minimal representations specifically useful are $k$-orthogonal representations 
\begin{equation*}
    \bm{X} = [\bm{U}_1, \ldots \bm{U}_{k-1}, \bm{G}_k, \bm{V}_{k+1}, \ldots, \bm{V}_d]
\end{equation*}
such that every $\bm{U}_i$ is left-orthogonal and every $\bm{V}_j$ is right-orthogonal
\begin{equation*}
    (U_i^L)^T U_i^L = I_{r_i}, \quad i = 1, \ldots, k-1, \quad V_j^R (V_j^R)^T = I_{r_{j-1}}, \quad j = k+1, \ldots, d.
\end{equation*}
We call $1$-orthogonal and $d$-orthogonal representations right- and left-orthogonal, respectively. A minimal $k$-orthogonal representation of a tensor can be computed with TT-SVD followed by a partial sweep of QR (or RQ) orthogonalizations.

The truncated TT-SVD algorithm can be used to approximate $\bm{X}$ with a tensor of given TT rank $\bm{r} \in \N^{d-1}$. Unlike the truncated SVD for matrices, the resulting approximation is not optimal but is quasi-optimal nonetheless
\begin{equation*}
    \| \text{TT-SVD}_{\bm{r}}(\bm{X}) - \bm{X} \|_F \leq \sqrt{d - 1} \| \text{opt}_{\bm{r}}(\bm{X}) - \bm{X} \|_F,
\end{equation*}
where $\text{opt}_{\bm{r}}(\bm{X})$ is the best rank-$\bm{r}$ approximation of $\bm{X}$ in the Frobenius norm.

\section{Our contributions}\label{section:contrib}


\subsection{Curvature bound}
For the needs of the convergence analysis, we are interested in estimating how quickly the projection operator onto the tangent space $\mathcal{P}_{T_{\bm{X}} \mathcal{M}_{\bm{r}}}$ changes as we move around on the manifold $\mathcal{M}_{\bm{r}}$. Another concern is the following. Every $\bm{X} \in \mathcal{M}_{\bm{r}}$ belongs to its own tangent space $\bm{X} \in T_{\bm{X}} \mathcal{M}_r$ but it is also important to know how well $\bm{X}$ can be approximated by other tangent spaces in its neighborhood, which essentially gives a bound on the curvature of the manifold. 

Our first result (Lemma~\ref{tt:lemma:curvature_bound}) is a new curvature bound for $\mathcal{M}_{\bm{r}}$. Denote by $\sigma_{\min}(\cdot)$ the smallest positive singular value of a matrix and, with some abuse of notation, the harmonic mean of the smallest positive singular values of the unfoldings of a tensor
\begin{equation*}
    \sigma_{\min}(\bm{X}) = (d-1)\left(\sum_{k=1}^{d-1} \frac{1}{\sigma_{\min}(X^{\langle k \rangle})}\right)^{-1}.    
\end{equation*}
\begin{lemma}\label{tt:lemma:curvature_bound}
For every pair of tensors $\bm{X},\bm{\Tilde{X}} \in \mathcal{M}_{\bm{r}}$ with the same TT ranks it holds that
\begin{equation*}
        \| (\mathrm{Id} - \mathcal{P}_{T_{\bm{\tilde{X}}} \mathcal{M}_{\bm{r}}}) \bm{X} \|_F \leq (d-1) \frac{\| \bm{X} - \bm{\tilde{X}} \|_F^2}{\sigma_{\min}(\bm{X})} \quad \text{and} \quad \| \mathcal{P}_{T_{\bm{X}} \mathcal{M}_{\bm{r}}} - \mathcal{P}_{T_{\bm{\tilde{X}}} \mathcal{M}_{\bm{r}}} \| \leq 2 (d-1) \frac{\| \bm{X} - \bm{\tilde{X}} \|_F}{\sigma_{\min}(\bm{X})}.
\end{equation*}
\end{lemma}
Similar bounds were obtained in \cite{LubichEtAlDynamical2013} (Lemma~4.5) for tensors in the hierarchical Tucker format, of which TT is a particular case. Most importantly, the bounds in Lemma~\ref{tt:lemma:curvature_bound} remain valid for \textit{any} tensor $\bm{\Tilde{X}} \in \mathcal{M}_{\bm{r}}$, while those in  \cite{LubichEtAlDynamical2013} hold only in a neighborhood of $\bm{X}$. Analogous global upper bounds were also derived in \cite{CaiEtAlProvable2021a} for $\| (\mathrm{Id} - \mathcal{P}_{T_{\bm{\tilde{X}}} \mathcal{M}_{\bm{r}}}) \bm{X} \|_F$ (Lemma~27) and for $\| \mathcal{P}_{T_{\bm{X}} \mathcal{M}_{\bm{r}}} - \mathcal{P}_{T_{\bm{\tilde{X}}} \mathcal{M}_{\bm{r}}} \|$ (Eq.~38). However, our bounds are tighter: 1) the constants are smaller; 2) we use the harmonic mean of the singular values, while \cite{LubichEtAlDynamical2013, CaiEtAlProvable2021a} work with the minimum of the singular values, and $\sigma_{\min}(\bm{X}) \geq \min_{k \in [d-1]} \sigma_{\min}(X^{\langle k \rangle})$. 

The two types of averaging coincide when all $\sigma_{\min}(X^{\langle k \rangle})$ are the same. For instance, if we take Lemma~\ref{tt:lemma:curvature_bound} with $d = 2$, we recover the curvature bounds for the matrix manifold as in \cite{WeiEtAlGuarantees2016} (Lemma~4.1). The situation is different when $d > 2$ and some of the unfoldings are ill-conditioned. This can occur when a full-rank tensor is approximated in the TT format with overestimated TT ranks. Assume that $\sigma_{\min}(X^{\langle k \rangle})$ are equal to $1$ for $d-1-s$ unfoldings and to $0 < \varepsilon < 1$ for the remaining $s$ unfoldings. We have
\begin{equation*}
    \min_{k \in [d-1]} \sigma_{\min}(X^{\langle k \rangle}) = \varepsilon, \quad \sigma_{\min}(\bm{X}) = 1 - \frac{s(1 - \varepsilon)}{s + (d-1-s)\varepsilon},
\end{equation*}
and $\sigma_{\min}(\bm{X})$ can be seen as a convex combination of $1$ and $\varepsilon$ with a dimension-dependent coefficient $0 < \alpha_{s,d} < 1$:
\begin{equation*}
    \sigma_{\min}(\bm{X}) = (1 - \alpha_{s,d}) + \alpha_{s,d} \varepsilon, \quad \alpha_{s,d} = \frac{s}{s + (d-1-s)\varepsilon}.
\end{equation*}
Lemma~\ref{tt:lemma:curvature_bound} shows that the curvature is tolerant to ill-conditioned unfoldings for high-dimensional tensors, while the previous results overestimate it. For example, take $\varepsilon = 10^{-2}$ and $d = 100$. Such high-dimensional tensors appear when the quantized TT format is used to approximate differential operators \cite{KazeevKhoromskijLow2012} and solve differential equations \cite{DolgovEtAlFast2012}. We get $\sigma_{\min}(\bm{X}) \approx 0.17$ for $s = 5$ and $\sigma_{\min}(\bm{X}) \approx 0.09$ for $s = 10$, which are about $d/s$ times larger than $\varepsilon$. As a result, the previous curvature bounds $(d - 1) / \varepsilon$ go down to about $s / \varepsilon$. If we let $d$ grow with $s$ and $\varepsilon$ fixed, the asymptotics are
\begin{equation*}
    \sigma_{\min}(\bm{X}) = 1 - \frac{s}{d} (\varepsilon^{-1} - 1) + O\left( \frac{1}{d^2} \right).
\end{equation*}


\subsection{Local convergence of Riemannian gradient descent}
\subsubsection{Tensor recovery}
The TT completion problem is a particular instance of a more general TT recovery problem with a linear measurement operator $\mathcal{R} : \Real^{n_1 \times \ldots \times n_d} \to \Real^{s}$, where one needs to recover a tensor $\bm{A}$ in the TT format given the measurements $\mathcal{R} \bm{A}$:
\begin{equation*}
    \| \mathcal{R} \bm{X} - \mathcal{R} \bm{A} \|_2^2 \to \min \quad \text{s.t.} \quad \bm{X} \in \mathcal{M}_{\bm{r}}.
\end{equation*}
This Riemannian optimization problem can be solved with the RGD, and to explicitly formulate the method, we need to choose the step size and the retraction mapping \cite{AbsilEtAlOptimization2009}. The truncated TT-SVD is a valid retraction on $\mathcal{M}_{\bm{r}}$ (see \cite{SteinlechnerRiemannian2016}), hence our RGD step is
\begin{equation}
\label{recov:eq:rgd}
    \bm{X}_{t+1} = \text{TT-SVD}_{\bm{r}} \left( \bm{X}_t - \alpha_t \bm{Y}_t \right) \in \mathcal{M}_{\bm{r}}, \quad \bm{Y}_t = \mathcal{P}_{\bm{X}_t} \mathcal{R}^* [\mathcal{R}\bm{X}_t - \mathcal{R}\bm{A}] \in T_{\bm{X}_t} \mathcal{M}_{\bm{r}},
\end{equation}
where we use $\mathcal{P}_{\bm{X}_t}$ as an alias for $\mathcal{P}_{T_{\bm{X}_t} \mathcal{M}_{\bm{r}}}$ and the step size is chosen via exact line search in the tangent space $T_{\bm{X}_t} \mathcal{M}_{\bm{r}}$:
\begin{equation*}
\label{recov:eq:step_size}
    \alpha_t = \| \bm{Y}_t \|_F^2 / \| \mathcal{R} \bm{Y}_t \|_F^2.
\end{equation*}
The appeal of this step size is in its closed-form formula, which greatly simplifies the analysis of the RGD \eqref{recov:eq:rgd}. From the numerical perspective, such $\alpha_t$ can be used as a good starting point for the backtracking scheme applied along the geodesic \cite{VandereyckenLowRank2013}; typically, though, $\alpha_t$ itself is sufficient \cite{SteinlechnerRiemannian2016}.

In general tensor recovery problems, the measurement operator $\mathcal{R}$ is assumed to exhibit a more standard (than \eqref{intro:eq:rip}) variant of the RIP. We will say that $\mathcal{R}$ satisfies the \textit{standard RIP} of order $\bm{r}$ if the following two-sided bound \cite{RauhutEtAlLow2017}
\begin{equation}
\label{recov:eq:rip_vanilla}
    (1 - \delta_{\bm{r}}) \| \bm{X} \|_F^2 \leq \| \mathcal{R} \bm{X} \|_2^2 \leq (1 + \delta_{\bm{r}}) \| \bm{X} \|_F^2
\end{equation}
holds for all tensors $\bm{X}$ of TT rank at most $\bm{r}$ with a RIP constant $0 < \delta_{\bm{r}} < 1$. An example of a measurement operator for which the standard RIP \eqref{recov:eq:rip_vanilla} holds with high probability are i.i.d. random Gaussian measurements \cite{RauhutEtAlLow2017}. The sampling operator $\mathcal{R}_{\Omega}$ of tensor completion, however, cannot fulfill the standard RIP \eqref{recov:eq:rip_vanilla} for \textit{all} tensors with low TT ranks (consider a sparse tensor). Nonetheless, the RGD convergence analyses for TT recovery and TT completion are very similar, and the proof of the former (which is slightly easier) can be easily adapted to fit the latter. This brings us to the new Theorem~\ref{recov:theorem:convergence}, which establishes local linear convergence of the RGD \eqref{recov:eq:rgd} for the TT recovery problem.

\begin{theorem}
\label{recov:theorem:convergence}
Let $\bm{A} \in \mathcal{M}_{\bm{r}}$ be a tensor of TT rank $\bm{r}$. Suppose that the measurement operator $\mathcal{R}$ satisfies the standard RIP \eqref{recov:eq:rip_vanilla} of order $2\bm{r}$ with a RIP constant $0 < \delta_{2\bm{r}} < 1$ and is bounded $\| \mathcal{R}^* \mathcal{R} \| \leq C$. Then the error on the current step of the RGD \eqref{recov:eq:rgd} is estimated via the previous error
\begin{equation*}
    \| \bm{X}_{t+1} - \bm{A} \|_F \leq \beta_t \| \bm{X}_t - \bm{A} \|_F
\end{equation*}
with a constant
\begin{equation*}
    \beta_t = (1 + \sqrt{d-1}) \left[ \frac{2\delta_{2\bm{r}}}{1 - \delta_{2\bm{r}}} + \left( 1 + \frac{C}{1 - \delta_{2\bm{r}}} \right) (d-1) \frac{\| \bm{X}_t - \bm{A} \|_F}{\sigma_{\min}(\bm{A})} \right].
\end{equation*}
If $\delta_{2\bm{r}} < (3 + 2 \sqrt{d-1})^{-1}$ and the initial condition $\bm{X}_0 \in \mathcal{M}_{\bm{r}}$ satisfies
\begin{equation*}
    (d-1) \frac{\| \bm{X}_0 - \bm{A} \|_F}{\sigma_{\min}(\bm{A})} < \frac{1}{1 + C - \delta_{2\bm{r}}} \left( \frac{1 - \delta_{2\bm{r}}}{1 + \sqrt{d-1}} - 2\delta_{2\bm{r}} \right), 
\end{equation*}
the iterations of RGD converge linearly to $\bm{A}$ at a rate
\begin{equation*}
    \| \bm{X}_{t+1} - \bm{A} \|_F < \beta_0^{t+1} \| \bm{X}_0 - \bm{A} \|_F, \quad \beta_0 < 1.
\end{equation*}
If $\mathcal{R}$ satisfies the standard RIP \eqref{recov:eq:rip_vanilla} of order $3\bm{r}$, the above results remain valid when $C$ is replaced with $1 + \delta_{3\bm{r}}$.
\end{theorem}

The novelty of our Theorem \ref{recov:theorem:convergence} is that we require the standard RIP \eqref{recov:eq:rip_vanilla} of order $2 \bm{r}$, while an analogous result from \cite{RauhutEtAlTensor2015} (Theorem 3) uses order $3 \bm{r}$. We achieve this by leveraging the upper bound $\| \mathcal{R}^* \mathcal{R} \| \leq C$. In addition, we consider a varying step size (\cite{RauhutEtAlTensor2015} sets $\alpha_t = 1$ for the proof) and provide explicit expressions for the radius of convergence and the convergence rate. We have not seen results similar to Theorem \ref{recov:theorem:convergence} with the standard RIP \eqref{recov:eq:rip_vanilla} of order $2 r$ in the matrix case, either.

\subsubsection{Tensor completion}
Turning to the tensor completion problem, we introduce a collection of multi-indices $\Omega \subset [n_1] \times \ldots \times [n_d]$ and denote by $\rho = |\Omega| / (n_1 \ldots n_d)$ the density of known elements. We define the sampling operator $\mathcal{R}_{\Omega} : \Real^{n_1 \times \ldots \times n_d} \to \Real^{n_1 \times \ldots \times n_d}$ as
\begin{equation*}
    \mathcal{R}_{\Omega} \bm{X} = \sum_{\omega \in \Omega} \bm{X}(\omega) \bm{E}_{\omega}, \quad \omega = (i_1, \ldots, i_d), 
\end{equation*}
where $\bm{E}_\omega = e_{i_1} \circ \ldots \circ e_{i_d}$ are canonical basis tensors. This definition allows $\Omega$ to contain repeated elements so in general $\mathcal{R}_{\Omega}$ is not a projection operator. It is, however, self-adjoint and positive semi-definite. For the ease of presentation, we will use $\mathcal{R} = \sqrt{\mathcal{R}_{\Omega}}$ as the measurement operator to reformulate the RGD \eqref{recov:eq:rgd} for the specific case of tensor completion:
\begin{equation}
\label{complt:eq:rgd}
    \bm{X}_{t+1} = \text{TT-SVD}_{\bm{r}} \left( \bm{X}_t - \alpha_t \bm{Y}_t \right) \in \mathcal{M}_{\bm{r}}, \quad \bm{Y}_t = \mathcal{P}_{\bm{X}_t} [\mathcal{R}_\Omega \bm{X}_t - \mathcal{R}_\Omega \bm{A}] \in T_{\bm{X}_t} \mathcal{M}_{\bm{r}},
\end{equation}
with the step size
\begin{equation*}
\label{complt:eq:step_size}
    \alpha_t = \frac{\| \bm{Y}_t \|_F^2}{\langle \mathcal{R}_\Omega \bm{Y}_t, \bm{Y}_t \rangle_F}.
\end{equation*}

As we discussed previously, the sampling operator cannot satisfy the standard RIP \eqref{recov:eq:rip_vanilla}, so we resort to a weaker (see Lemma~\ref{recov:lemma:rip_rip}) assumption, which is just a verbatim translation of the RIP \eqref{intro:eq:rip} from matrix completion to the multi-dimensional setting:
\begin{equation}
\label{complt:eq:rip}
    \| \mathcal{P}_{T_{\bm{A}} \mathcal{M}_{\bm{r}}} - \rho^{-1} \mathcal{P}_{T_{\bm{A}} \mathcal{M}_{\bm{r}}} \mathcal{R}_{\Omega} \mathcal{P}_{T_{\bm{A}} \mathcal{M}_{\bm{r}}} \| < \varepsilon.
\end{equation}
Armed with this assumption, we can prove a new Theorem \ref{complt:theorem:convergence} on the convergence of the RGD \eqref{complt:eq:rgd} for TT completion.

\begin{theorem}
\label{complt:theorem:convergence}
Let $\bm{A} \in \mathcal{M}_{\bm{r}}$ be a tensor of TT rank $\bm{r}$. Suppose that the sampling operator $\mathcal{R}_\Omega$ satisfies the RIP \eqref{complt:eq:rip} and is bounded $\| \mathcal{R}_{\Omega} \| \leq C$. Then the error on the current step of the RGD \eqref{complt:eq:rgd} is estimated via the previous error
\begin{equation*}
    \| \bm{X}_{t+1} - \bm{A} \|_F \leq \beta_t \| \bm{X}_t - \bm{A} \|_F
\end{equation*}
with a constant
\begin{equation*}
    \beta_t = (1 + \sqrt{d-1}) \left[ \frac{2 \varepsilon_t }{1 - \varepsilon_t} + \left( 1 + \frac{C}{1 - \varepsilon_t} \right) (d-1) \frac{\| \bm{X}_t - \bm{A} \|_F}{\sigma_{\min}(\bm{A})} \right], \quad \varepsilon_t = \varepsilon + \left( 2 + 4 C \rho^{-1}\right) (d-1) \frac{\| \bm{X}_t - \bm{A} \|_F}{\sigma_{\min}(\bm{A})}.
\end{equation*}
If $\varepsilon < (3 + 2 \sqrt{d-1})^{-1}$ and the initial condition $\bm{X}_0 \in \mathcal{M}_{\bm{r}}$ satisfies
\begin{equation*}
    (d-1) \frac{\| \bm{X}_0 - \bm{A} \|_F}{\sigma_{\min}(\bm{A})} < \min\left( \frac{1 - \varepsilon}{2 + 4 C \rho^{-1}}, \left( 5 + C + 8C\rho^{-1} + \frac{2 + 4C\rho^{-1}}{1 + \sqrt{d-1}} - \varepsilon \right)^{-1} \left( \frac{1 - \varepsilon}{1 + \sqrt{d-1}} - 2\varepsilon \right) \right),
\end{equation*}
the iterations of RGD converge linearly to $\bm{A}$ at a rate
\begin{equation*}
    \| \bm{X}_{t+1} - \bm{A} \|_F < \beta_0^{t+1} \| \bm{X}_0 - \bm{A} \|_F, \quad \beta_0 < 1.
\end{equation*}
\end{theorem}
To the best of our knowledge, local linear convergence of the RGD \eqref{complt:eq:rgd} has not been established before. In \cite{CaiEtAlProvable2021a}, Riemannian TT completion is addressed from the theoretical point of view as well, but the algorithm is different there: an additional trimming procedure is applied on every iteration before TT-SVD to ensure that all the elements of the tensor remain below a certain threshold. This algorithm is proved to locally linearly converge in \cite{CaiEtAlProvable2021a} (Lemmas 5 and 9), but the assumptions are stricter than in our Theorem \ref{complt:theorem:convergence}: in addition to the RIP \eqref{complt:eq:rip} and the bound $\| \mathcal{R}_{\Omega} \| \leq C$ (which are not present in the formulations of the Lemmas, but can be found in the proof of Lemma 9 under the names $\bm{\mathcal{E}}_1$ and $\bm{\mathcal{E}}_2$), the initial condition $\bm{X}_0$ is required to have low \textit{interface coherence} (we will talk about this notion later on). Meanwhile, our Theorem \ref{complt:theorem:convergence} guarantees local linear convergence for \textit{any} initial condition as long as it is sufficiently close to $\bm{A}$.

We can also compare the implications of Theorem \ref{complt:theorem:convergence} for matrix completion with Theorem \ref{intro:theorem:wei_riemann}. Our result guarantees convergence when the RIP \eqref{intro:eq:rip} holds with a larger $\varepsilon$ ($1/5$ against $1/22$) and, as a consequence, when fewer elements of the matrix are known, owing to Theorem \ref{intro:theorem:recht_rip}. 

From the numerical perspective, the trimming step in \cite{CaiEtAlProvable2021a} renders the whole algorithm expensive both in terms of memory requirements and computational complexity, since a full tensor needs to be assembled from its TT representation and then TT-SVD is applied to a full tensor too. It is noted, however, that in numerical experiments the iterations with and without trimming behave in a nearly identical manner. While the algorithm without trimming (which is exactly our RGD \eqref{complt:eq:rgd}) is much more efficient, there is still a question of how to choose the initial point $\bm{X}_0$: the sequential spectral initialization of \cite{CaiEtAlProvable2021a} (Alg. 3) and the possible multi-dimensional extensions of the initialization strategies in \cite{WeiEtAlGuarantees2016} may be provable, but their computational complexities are likely to dwarf the resources needed to carry out the RGD iterations for large tensors. We leave aside the problem of choosing the initial point in this paper, but we believe that a more promising direction that can lead to a provably convergent computationally efficient method is random initialization \cite{ChenEtAlGradient2019, MaEtAlImplicit2020}.

It should be noted that in Theorems \ref{recov:theorem:convergence} and \ref{complt:theorem:convergence}, we implicitly assume that the sequences generated by the RGD always remain on the manifold $\mathcal{M}_{\bm{r}}$; however, though virtually unseen in practice, the TT ranks can become smaller. This phenomenon was studied in the matrix case for a projected line search method on the algebraic variety of matrices with rank not bigger (as opposed to equal) than a certain fixed value \cite{SchneiderUschmajewConvergence2015}.


\subsection{Recovery guarantees}
The main assumption in Theorem \ref{complt:theorem:convergence} that guarantees local linear convergence of the RGD \eqref{complt:eq:rgd} is that the sampling operator $\mathcal{R}_{\Omega}$ satisfies the RIP \eqref{complt:eq:rip}. Assuming that the indices $\Omega$ are chosen uniformly at random with replacement, we want to derive probabilistic sufficient conditions which ensure that the RIP holds with high probability. In this setting, we can also obtain a new bound on the sampling operator.

\begin{lemma}
\label{guarantee:lemma:repetitions}
Let $\Omega \subset [n_1] \times \ldots \times [n_d]$ be a collection of indices sampled uniformly at random with replacement. Then the norm of the sampling operator is bounded by
\begin{equation*}
    \| \mathcal{R}_{\Omega} \| \leq \frac{d \beta}{w(d)} \log (n), \quad n = \max( n_1, \ldots, n_d),
\end{equation*}
with probability at least $1 - n^{d(1-\beta)}$ for $n \geq 16$ and $\beta > 1$. Here,  $w(d)$ is the principal branch of the Lambert W function, also known as product logarithm.
\end{lemma}
\begin{proof}
The norm $\| \mathcal{R}_{\Omega} \|$ is nothing but the maximum number of repetitions in the sample. Consider $|\Omega|$ i.i.d. Bernoulli random variables $\xi_j$ with probability of success $1 / (n_1 \ldots n_d)$ and let $\xi = \sum_{j} \xi_j$. Since all the indices in $\Omega$ are drawn with equal probability with replacement, $\xi$ describes how many times a single fixed entry is sampled. Then the probability of it being sampled more than $k$ times can be upper bounded with the help of the Chernoff bound
\begin{equation*}
    \Prob \left\{ \xi > x \right \} \leq \left( \frac{\rho}{x} \right)^x \exp(x - \rho), \quad \rho = \frac{|\Omega|}{n_1 \ldots n_d}.
\end{equation*}
The union bound over all the entries leads to
\begin{equation*}
    \Prob \left\{ \| \mathcal{R}_{\Omega} \| > x \right\} \leq (n_1 \ldots n_d) \Prob \left\{ \xi > x \right \} \leq n^d \left( \frac{\rho}{x} \right)^x \exp(x - \rho) < n^d \left( \frac{1}{x} \right)^x \exp(x).
\end{equation*}
It remains to substitute $x = d\beta \log(n) / w(d)$ and note that for $n \geq 16 > \exp(e)$,
\begin{equation*}
    w(d) \exp(w(d)) = d \leq \frac{\log(n)}{e} d < \frac{\log(n)}{e} d \beta. \qedhere
\end{equation*}
\end{proof}
Lemma~\ref{guarantee:lemma:repetitions} is a direct multi-dimensional extension of \cite{RechtSimpler2011}, Proposition 5. An analogous result appears in \cite{CaiEtAlProvable2021a} (Lemma~33) with a $d \beta \log(n)$ bound. The bound we prove is tighter, especially for large $d$. Indeed, the Lambert W function behaves as $w(d) = \log(d) - \log(\log(d)) + o(1)$, so the bound grows as $d \log(n) / \log(d)$. A similar asymptotic was mentioned in \cite{RechtSimpler2011}.

Going back to the RIP \eqref{complt:eq:rip}, we want to extend Theorem \ref{intro:theorem:recht_rip} from matrices to tensors in the TT format. To this end, we need to generalize the assumption of bounded coherence \eqref{intro:eq:incoherent} and/or the notion of coherence \eqref{intro:eq:coherence} itself. Every matrix, seen as a tensor, coincides with its unfolding $A = A^{\langle 1 \rangle}$ and has its column and row spaces spanned by the columns of the interface matrices $A_{\leq 1}$ and $A_{\geq 2}$, respectively. The incoherence assumption \eqref{intro:eq:incoherent} can then be written in a way that is easily extended to the multi-dimensional case:
\begin{equation*}
    \mu(A_{\leq 1}) \leq \mu_0, \quad \mu(A_{\geq 2}) \leq \mu_0.
\end{equation*}
We define the \textit{interface coherence} of a tensor $\bm{A}$ as the maximum coherence of its left and right interface matrices:
\begin{equation}
\label{guarantee:eq:interface_coherence}
    \mu_{I}(\bm{A}) = \max\Big( \mu(A_{\leq 1}), \mu(A_{\geq 2}), \ldots, \mu(A_{\leq d-1}), \mu(A_{\geq d}) \Big).
\end{equation}
Recalling the definition of the coherence \eqref{intro:eq:coherence}, we get
\begin{align*}
    \mu(A_{\leq k}) &= \frac{n_1 \ldots n_k}{r_k} \max_{i_1 \in [n_1], \ldots, i_k \in [n_k]} \| P_{\leq k} (e_{i_k} \otimes \ldots \otimes e_{i_1}) \|^2_2, \\
    \mu(A_{\geq k+1}) &= \frac{n_{k+1} \ldots n_d}{r_k} \max_{i_{k+1} \in [n_{k+1}], \ldots, i_d \in [n_d]} \| P_{\geq k+1} (e_{i_{k+1}} \otimes \ldots \otimes e_{i_d}) \|^2_2.
\end{align*}
As we replace the incoherence assumption \eqref{intro:eq:incoherent} with the interface incoherence, we can prove an analog of Theorem \ref{intro:theorem:recht_rip} for tensors with low TT ranks.

\begin{theorem}
\label{guarantee:theorem:rip_interface}
Let $\bm{A} \in \mathcal{M}_{\bm{r}}$ be a tensor of TT rank $\bm{r}$ with bounded interface coherence $\mu_I(\bm{A}) \leq \mu_0$ and let $\Omega \subset [n_1] \times \ldots \times [n_d]$ be a collection of indices sampled uniformly at random with replacement. Then the RIP \eqref{complt:eq:rip}
\begin{equation*}
    \| \mathcal{P}_{T_{\bm{A}} \mathcal{M}_{\bm{r}}} - \rho^{-1} \mathcal{P}_{T_{\bm{A}} \mathcal{M}_{\bm{r}}} \mathcal{R}_{\Omega} \mathcal{P}_{T_{\bm{A}} \mathcal{M}_{\bm{r}}} \| < \varepsilon, \quad \rho = \frac{|\Omega|}{n_1 \ldots n_d},
\end{equation*}
holds with probability at least $1 - 2n^{d(1 - \beta)}$, $n = \max( n_1, \ldots, n_d)$, for all $\beta > 1$ provided that
\begin{equation*}
    |\Omega| \geq \frac{8}{3} \frac{\beta}{\varepsilon^2} \mu_0 \left(n_1 r_1 + \mu_0 \sum_{k = 2}^{d-1} r_{k - 1} n_k r_k + r_{d-1}n_d \right) d \log(n).
\end{equation*}
\end{theorem}

The interface coherence \eqref{guarantee:eq:interface_coherence} can also be found in \cite{CaiEtAlProvable2021a} under the name $\text{Incoh}(\bm{A})$ (to be precise, $\text{Incoh}(\bm{A}) = \sqrt{\mu_I(\bm{A})}$). A very similar result was proved independently there (Lemma 31). The differences are minor: we treat $\varepsilon$ as a parameter and our estimate is more detailed. In the two-dimensional case, Theorem \ref{guarantee:theorem:rip_interface} becomes exactly \cite{RechtSimpler2011}, Theorem 6. 

Theorem \ref{guarantee:theorem:rip_interface}, coupled with Theorem \ref{complt:theorem:convergence} and Lemma \ref{guarantee:lemma:repetitions}, shows that, with high probability, the RGD \eqref{complt:eq:rgd} converges locally to $\bm{A}$ when the number of elements in the sample is of order
\begin{equation*}
    |\Omega| \gtrsim \mu_0^2 d^2 r^2 n \log(n),
\end{equation*}
where $n = \max(n_1, \ldots, n_d)$ and $r = \max(r_1, \ldots, r_{d-1})$.  Every tensor of TT rank $\bm{r}$ is described with $O(dnr^2)$ parameters, so the local recovery is highly probable with $d\log(n)$ oversampling, just as in the matrix case. To compare, the algorithm in \cite{CaiEtAlProvable2021a} (Lemma 5) is proved to converge locally when
\begin{equation*}
    |\Omega| \gtrsim C_d \mu_0^{\frac{d}{2}} r^{\frac{d}{2}} n^{\frac{d}{2}} \log^{d}(n), \quad C_d = C_d(d).
\end{equation*}
The problem, however, is that for a tensor $\bm{A}$ with minimal TT representation $\bm{A} = [\bm{G}_1, \ldots, \bm{G}_d]$, the interface matrices are intimately interconnected (see Eq.~\eqref{eq:tt:interface_rec}),
\begin{equation*}
    A_{\leq k} = (I_{n_k} \otimes A_{\leq k-1}) G_k^L,
\end{equation*}
and so their coherences are also far from being independent. Moreover, $\mu(A_{\leq d-1})$ and $\mu(A_{\geq 2})$ can become as high as $n^{d-1}/r$ and hence the value of the interface coherence $\mu_{I}(\bm{A})$ is a source of potential problems for the sample complexity.

In defining interface coherence, we were inspired by a particular way to express incoherence for matrices, via interface matrices. Here we draw a different analogy. Let $A = [\bm{G}_1, \bm{G}_2]$ be a minimal representation of a matrix. Since their left and right unfoldings satisfy (see Eq.~\eqref{eq:tt:interface_rec})
\begin{equation*}
    A = G_1^L (G_2^R)^T,
\end{equation*}
we can rewrite the incoherence assumption \eqref{intro:eq:incoherent} as
\begin{equation*}
    \mu(G_1^L) \leq \mu_0, \quad \mu((G_2^R)^T) \leq \mu_0.
\end{equation*}
We will try to extend the notion of coherence to tensors through the TT cores.

Let $\bm{U} \in \Real^{r \times n \times s}$ be a three-dimensional left-orthogonal tensor. Denote by $U^{(i)} \in \Real^{r \times s}$ the $i$-th subblock of its left unfolding:
\begin{equation*}
    U^L =
    \begin{bmatrix}
        U^{(1)} \\
        \vdots \\
        U^{(n)}
    \end{bmatrix} \in \Real^{rn \times s}.
\end{equation*}
We define the \textit{left coherence of a three-dimensional left-orthogonal tensor} as
\begin{equation}
\label{guarantee:eq:left_coherence}
    \mu_{L}(\bm{U}) = \frac{rn}{s} \max_{i \in [n]} \| U^{(i)} \|_2^2.
\end{equation}
When $r = 1$, the tensor $\bm{U}$ becomes a matrix, the subblocks $U^{(i)}$ become rows, their spectral norm equals to the Euclidean norm, and we recognize that the left coherence is just the coherence of a matrix with orthonormal columns.

Likewise, let $\bm{V} \in \Real^{r \times n \times s}$ be a right-orthogonal tensor and let $(V^{(i)})^T \in \Real^{r \times s}$ be the $i$-th subblock of the right unfolding:
\begin{equation*}
    V^R =
    \begin{bmatrix}
        (V^{(1)})^T & \hdots & (V^{(n)})^T
    \end{bmatrix} \in \Real^{r \times ns}.
\end{equation*}
We define the \textit{right coherence of a three-dimensional right-orthogonal tensor} $\bm{V}$ as
\begin{equation}
\label{guarantee:eq:right_coherence}
    \mu_R(\bm{V}) = \frac{sn}{r} \max_{i \in [n]} \| V^{(i)} \|_2^2.
\end{equation}
In complete analogy, the right coherence of a three-dimensional right-orthogonal tensor becomes the coherence of a (transposed) matrix with orthonormal rows when $s = 1$. Note that while we defined the coherence \eqref{intro:eq:coherence} for arbitrary matrices, the notions of left and right coherences require orthogonality.

Now, consider a $d$-dimensional tensor $\bm{X}$ in a minimal left-orthogonal TT representation $\bm{X} = [\bm{U}_1, \ldots, \bm{U}_{d-1}, \bm{G}_d]$. Since the first $d-1$ TT cores are left-orthogonal, we can compute their left coherences $\{ \mu_L(\bm{U}_k) \}_{k \in [d-1]}$. But do these values characterize the specific TT representation of $\bm{X}$ or the tensor itself? What happens to them when we choose a different minimal left-orthogonal TT representation? The following Lemma~\ref{guarantee:lemma:core_coherence} shows that $\{ \mu_L(\bm{U}_k) \}_{k \in [d-1]}$ do not change.
\begin{lemma}
\label{guarantee:lemma:core_coherence}
Let $\bm{X} = [\bm{U}_1, \ldots, \bm{U}_{d-1}, \bm{G}_d] = [\bm{\tilde{U}}_1, \ldots, \bm{\tilde{U}}_{d-1}, \bm{\tilde{G}}_d]$ be two minimal left-orthogonal TT representations. Then the left coherences of their TT cores coincide:
\begin{equation*}
    \mu_L(\bm{U}_k) = \mu_L(\bm{\tilde{U}}_k), \quad k \in [d-1].
\end{equation*}
The same is true for any two right-orthogonal TT representations and the right coherences of their TT cores.
\end{lemma}
\begin{proof}
We carry out the proof for the left coherences. Consider the column span of the first interface matrix $X_{\leq 1}$. It is spanned by two orthonormal bases $U_1^L$ and $\tilde{U}_1^L$, so there exists an orthogonal matrix $Q_1 \in \Real^{r_1 \times r_1}$ such that $\tilde{U}_1^L = U_1^L Q_1$ and
\begin{equation*}
    \mu_L(\bm{\tilde{U}}_1) = \frac{r_0 n_1}{r_1} \max_{i \in [n_1]} \| \tilde{U}^{(i)} \|_2^2 = \frac{r_0 n_1}{r_1} \max_{i \in [n_1]} \| U^{(i)} Q_1 \|_2^2 = \frac{r_0 n_1}{r_1} \max_{i \in [n_1]} \| U^{(i)} \|_2^2 = \mu_L(\bm{U}_1).
\end{equation*}
By factoring $Q_1$ out of the first TT core and attaching it to the second TT core as
\begin{equation*}
    \tilde{U}_1^L \mapsto U_1^L, \quad \tilde{U}_2^L \mapsto \hat{U}_2^L = (I_{n_2} \otimes Q_1)\tilde{U}_2^L = 
    \begin{bmatrix}
        Q_1 \tilde{U}_2^{(1)} \\
        \vdots \\
        Q_1 \tilde{U}_2^{(n_2)}
    \end{bmatrix}
\end{equation*}
we get a new minimal left-orthogonal TT representation with $\mu_L( \bm{\hat{U}}_2 ) = \mu_L( \bm{\tilde{U}}_2 )$:
\begin{equation*}
    \bm{X} = [\bm{U}_1, \bm{\hat{U}}_2, \bm{\tilde{U}}_3, \ldots, \bm{\tilde{U}}_{d-1}, \bm{\tilde{G}}_d].
\end{equation*}
Now suppose we have a minimal left-orthogonal TT representation with $\mu_L( \bm{\hat{U}}_{k} ) = \mu_L( \bm{\tilde{U}}_{k} )$:
\begin{equation*}
    \bm{X} = [\bm{U}_1, \ldots, \bm{U}_{k-1}, \bm{\hat{U}}_{k}, \bm{\tilde{U}}_{k+1}, \ldots, \bm{\tilde{U}}_{d-1}, \bm{\tilde{G}}_d].
\end{equation*}
As the recursive formulas for the interface matrices \eqref{eq:tt:interface_rec} show, the column space of $X_{\leq k}$ is spanned by two orthonormal bases that are related via an orthogonal matrix $Q_k \in \Real^{r_k \times r_k}$ so that
\begin{equation*}
    (I_{n_k} \otimes U_{\leq k-1}) U_k^L = (I_{n_k} \otimes U_{\leq k-1}) \hat{U}_k^L Q_k.
\end{equation*}
Since $U_{\leq k-1}^T U_{\leq k-1} = I_{r_{k-1}}$ we get $U_k^L = \hat{U}_k^L Q_k$ and $\mu_L(\bm{U}_k) = \mu_L( \bm{\hat{U}}_k ) = \mu_L( \bm{\tilde{U}}_k )$. Attaching $Q_k$ to the next TT core gives a new minimal left-orthogonal TT representation
\begin{equation*}
    \bm{X} = [\bm{U}_1, \ldots, \bm{U}_{k}, \bm{\hat{U}}_{k+1}, \bm{\tilde{U}}_{k+2}, \ldots, \bm{\tilde{U}}_{d-1}, \bm{\tilde{G}}_d]
\end{equation*}
with $\mu_L( \bm{\hat{U}}_{k+1} ) = \mu_L( \bm{\tilde{U}}_{k+1} )$ if $k \leq d-2$ and $\bm{\hat{G}}_d = \bm{G}_d$ if $k = d-1$.
\end{proof}

Since $\{ \mu_L(\bm{U}_k) \}_{k \in [d-1]}$ are a property of the tensor, rather than its TT representation, this motivates us to define a new notion of \textit{core coherence of a tensor}.

\begin{definition}\label{guarantee:def:core_coherence}
Let $\bm{X} \in \Real^{n_1 \times \ldots \times n_d}$ be a tensor with minimal left- and right-orthogonal TT representations
\begin{equation*}
    \bm{X} = [\bm{U}_1, \ldots, \bm{U}_{d-1}, \bm{G}_d] = [\bm{G}_1, \bm{V}_2, \ldots, \bm{V}_d].
\end{equation*}
The \textit{$k$-th left core coherence $\mu_L^{(k)}(\bm{X})$ of $\bm{X}$} is defined as the left coherence \eqref{guarantee:eq:left_coherence} of the $k$-th TT core of its minimal left-orthogonal TT representation:
\begin{equation*}
    \mu_L^{(k)}(\bm{X}) = \mu_L(\bm{U}_k), \quad k \in [d-1].
\end{equation*}
The \textit{$(k+1)$-th right core coherence $\mu_R^{(k+1)}(\bm{X})$ of $\bm{X}$} is defined as the right coherence \eqref{guarantee:eq:right_coherence} of the $(k+1)$-th TT core of its minimal right-orthogonal TT representation:
\begin{equation*}
    \mu_R^{(k+1)}(\bm{X}) = \mu_R(\bm{V}_{k+1}), \quad k \in [d-1].
\end{equation*}
The \textit{core coherence $\mu_C(\bm{X})$ of $\bm{X}$} is defined as the maximum of its left and right core coherences:
\begin{equation}
\label{guarantee:eq:core_coherence}
    \mu_C(\bm{X}) = \max \Big( \mu_L^{(1)}(\bm{X}), \ldots, \mu_L^{(d-1)}(\bm{X}), \mu_R^{(2)}(\bm{X}), \ldots, \mu_R^{(d)}(\bm{X}) \Big).
\end{equation}
\end{definition}

When $X$ is a matrix, $\mu_L^{(1)}(X)$ is the coherence \eqref{intro:eq:coherence} of its column space and $\mu_R^{(2)}(X)$ is the coherence of its row space. The core coherence $\mu_C(X)$ and the interface coherence $\mu_I(X)$ coincide for $d = 2$ as well, but are very different for $d > 2$. The following Lemma~\ref{guarantee:lemma:core2interface} shows the relationship between the two in the multi-dimensional case.
\begin{lemma}
\label{guarantee:lemma:core2interface}
Let $\bm{A} \in \mathcal{M}_{\bm{r}}$ be a tensor of TT rank $\bm{r}$ with bounded core coherence $\mu_C(\bm{A}) \leq \mu_1$  Then the coherences of its left and right interface matrices are estimated as
\begin{equation*}
    \mu(A_{\leq k}) \leq \mu_1^k, \quad \mu(A_{\geq k+1}) \leq \mu_1^{d - k}, \quad k \in [d-1].
\end{equation*}
\end{lemma}
\begin{proof}
Consider a minimal left-orthogonal TT representation $\bm{A} = [\bm{U}_1, \ldots, \bm{U}_{d-1}, \bm{G}_d]$. The projection onto the column space of an interface matrix $P_{\leq k} = U_{\leq k} U_{\leq k}^T$ can be written, with the help of the recursive formulas \eqref{eq:tt:interface_rec}, as
\begin{equation*}
    U_{\leq 1} = U_1^L, \quad U_{\leq k} = (I_{n_k} \otimes U_{\leq k-1})U_{k}^L.
\end{equation*}
It follows that
\begin{equation*}
    U_{\leq k}^T(e_{i_k} \otimes \ldots \otimes e_{i_1}) = \left( U_{1}^{(i_1)} U_{2}^{(i_2)} \ldots U_{k}^{(i_k)} \right)^T \in \Real^{r_k}
\end{equation*}
and
\begin{align*}
    \| P_{\leq k} (e_{i_k} \otimes \ldots \otimes e_{i_1}) \|^2_2 &= \left\| \left( U_{1}^{(i_1)} U_{2}^{(i_2)} \ldots U_{k}^{(i_k)} \right)^T \right\|_2^2 \leq \| U_{1}^{(i_1)} \|_2^2 \ldots \| U_{k}^{(i_k)} \|_2^2\\
    &\leq \frac{r_1}{n_1} \frac{r_2}{r_1 n_2} \ldots \frac{r_k}{r_{k-1}n_k} \mu_1^k = \frac{r_k}{n_1 \ldots n_k} \mu_1^k.
\end{align*}
The proof is the same for the right interface matrices.
\end{proof}

We propose to replace the interface incoherence assumption in Theorem~\ref{guarantee:theorem:rip_interface} with a new core incoherence one, which leads to Theorem~\ref{guarantee:theorem:rip_core}.

\begin{theorem}\label{guarantee:theorem:rip_core}
Let $\bm{A} \in \mathcal{M}_{\bm{r}}$ be a tensor of TT rank $\bm{r}$ with bounded core coherence $\mu_C(\bm{A}) \leq \mu_1$ and let $\Omega \subset [n_1] \times \ldots \times [n_d]$ be a collection of indices sampled uniformly at random with replacement. Then the RIP \eqref{complt:eq:rip}
\begin{equation*}
    \| \mathcal{P}_{T_{\bm{A}} \mathcal{M}_{\bm{r}}} - \rho^{-1} \mathcal{P}_{T_{\bm{A}} \mathcal{M}_{\bm{r}}} \mathcal{R}_{\Omega} \mathcal{P}_{T_{\bm{A}} \mathcal{M}_{\bm{r}}} \| < \varepsilon, \quad \rho = \frac{|\Omega|}{n_1 \ldots n_d},
\end{equation*}
holds with probability at least $1 - 2n^{d(1 - \beta)}$, $n = \max( n_1, \ldots, n_d )$, for all $\beta > 1$ provided that
\begin{equation*}
    |\Omega| \geq \frac{8}{3} \frac{\beta}{\varepsilon^2} \mu_1^{d-1} \left( \sum_{k = 1}^{d} r_{k - 1} n_k r_k \right) d \log(n).
\end{equation*}
\end{theorem}

Lemma \ref{guarantee:lemma:core2interface} shows that, in the worst case, the interface coherence \eqref{guarantee:eq:interface_coherence} can be bounded by $\mu_1^{d-1}$ and, consequently, Theorem \ref{guarantee:theorem:rip_interface} (and \cite{CaiEtAlProvable2021a}, Lemma 31) gives the sample complexity of
\begin{equation*}
    |\Omega| \gtrsim \mu_1^{2d-2} d^2 r^2 n \log(n),
\end{equation*}
where $n = \max(n_1, \ldots, n_d)$ and $r = \max(r_1, \ldots, r_{d-1})$. Once we use the core coherence \eqref{guarantee:eq:core_coherence} directly, Theorem \ref{guarantee:theorem:rip_core} allows us to improve the estimate $\mu_1^{d-1}$ times:
\begin{equation*}
    |\Omega| \gtrsim \mu_1^{d-1} d^2 r^2 n \log(n).
\end{equation*}
Given this many elements of a tensor, Theorem \ref{complt:theorem:convergence} ensures, with high probability, that the RGD \eqref{complt:eq:rgd} converges locally to $\bm{A}$. When $d = 2$, Theorems \ref{guarantee:theorem:rip_interface} and \ref{guarantee:theorem:rip_core} are equivalent and repeat \cite{RechtSimpler2011} (Theorem 6).


\subsection{Tensor train completion with auxiliary subspace information}
We would also like to show how the core coherence \eqref{guarantee:eq:core_coherence} can be used in a different setting. As an example, we choose tensor completion with auxiliary subspace information. In this scenario, in addition to the elements of the tensor $\mathcal{R}_{\Omega} \bm{A}$, we know that the mode-$k$ fiber spans of $\bm{A}$ belong to particular low-dimensional subspaces. Such formulations appear, for example, in multi-label learning \cite{XuEtAlSpeedup2013} and bioinformatics \cite{NatarajanDhillonInductive2014, ChenEtAlPredicting2018}.

Namely, let matrices $Q_k \in \Real^{n_k \times m_k}$ have orthonormal columns that span the subspaces in question. If $m_k = n_k$, no extra information is given about the mode-$k$ fibers. The unknown tensor $\bm{A}$ can then be represented as
\begin{equation}\label{side:eq:subspaces}
    \bm{A} = \bm{B} \times_1 Q_1 \times_2 \ldots \times_d Q_d.
\end{equation}
This is a Tucker decomposition of $\bm{A}$ with the Tucker core $\bm{B} \in \Real^{m_1 \times \ldots \times m_d}$ and Tucker factors $Q_k$. Since $\bm{A} \in \mathcal{M}_{\bm{r}}$, we can also write its minimal TT representation $\bm{A} = [\bm{G}_1, \bm{G}_2, \ldots, \bm{G}_d]$, and it follows that $\bm{B}$ admits a minimal TT representation
\begin{equation*}
    \bm{B} = [\bm{S}_1, \bm{S}_2, \ldots, \bm{S}_d]
\end{equation*}
with TT cores $\bm{S}_k = \bm{G}_k \times_2 Q_k^T$ and the same TT ranks $\bm{r}$. Due to orthogonality, we also have $\bm{G}_k = \bm{S}_k \times_2 Q_k$.

There are then two ways to look at TT completion with subspace information. First, it is a usual TT completion problem for $\bm{A}$, where we know some of its elements $\mathcal{R}_{\Omega} \bm{A}$ and impose an additional constraint $\bm{G}_k = \bm{G}_k \times_2 Q_k Q_k^T$ on the TT cores. Second, we can treat it as a TT recovery problem for the small tensor $\bm{B}$ with a special measurement operator
\begin{equation*}
    \mathcal{R} \bm{B} = \mathcal{R}_{\Omega} \left( \bm{B} \times_1 Q_1 \times_2 \ldots \times_d Q_d \right).
\end{equation*}

Whatever the preferred point of view, the number of parameters that describe $\bm{A}$ is $O(d m r^2)$, where $m = \max(m_1, \ldots, m_d)$ and $r = \max(r_1, \ldots, r_{d-1})$. Therefore, it is reasonable to expect that the required number of known elements $|\Omega|$ should be reduced in the presence of auxiliary subspace information. And we prove the corresponding Theorem. Informally, it states that if $\bm{A}$ has bounded core coherence $\mu_C(\bm{A}) \leq \mu_1$ and the auxiliary subspaces have bounded coherences $\mu(Q_k) \leq \mu_2$, then, with high probability, a modified RGD converges locally to $\bm{A}$ when
\begin{equation*}
    |\Omega| \gtrsim \mu_1^{d-1} \mu_2 d^2 r^2 m \log(m).
\end{equation*}
This is the first theoretical result on the sample complexity of TT completion with auxiliary subspace information. We leave the more detailed and rigorous discussion for Appendix~\ref{section:side}.
\section{Proofs of main results}\label{section:proofs}
\subsection{Curvature bound}
To describe the tangent spaces to $\mathcal{M}_{\bm{r}}$, consider minimal left- and right-orthogonal TT representations of $\bm{X} \in \mathcal{M}_{\bm{r}}$ denoted by
\begin{equation*}
    \bm{X} = [\bm{U}_1, \ldots, \bm{U}_{d-1}, \bm{G}_d] = [\bm{G}_1, \bm{V}_2, \ldots, \bm{V}_d].
\end{equation*}
Every tangent vector $\bm{Y} \in T_{\bm{X}} \mathcal{M}_{\bm{r}}$ can be uniquely represented as a sum $\bm{Y} = \sum_{k=1}^d \bm{Y}_k$ with non-minimal TT representations \cite{SteinlechnerRiemannian2016}
\begin{equation*}
    \bm{Y}_k = [\bm{U}_1, \ldots, \bm{U}_{k-1}, \bm{\Upsilon}_k, \bm{V}_{k+1}, \ldots, \bm{V}_d],
\end{equation*}
where for $k \in [d-1]$ the TT cores $\bm{\Upsilon}_k \in \Real^{r_{k-1} \times n_k \times r_k}$ satisfy the gauge conditions for the left unfoldings
\begin{equation*}
    \left(U_k^L\right)^T \Upsilon_k^L = 0 \in \Real^{r_k \times r_k}.
\end{equation*}
The last TT core $\bm{\Upsilon}_d$ does not have a gauge condition. On introducing the subspaces 
\begin{gather*}
    T_k = \left\{ [\bm{U}_1, \ldots, \bm{U}_{k-1}, \bm{\Upsilon}_k, \bm{V}_{k+1}, \ldots, \bm{V}_d]~:~\bm{\Upsilon}_k \in \Real^{r_{k-1} \times n_k \times r_k},~ \left(U_k^L\right)^T \Upsilon_k^L = 0 \right\},\\
    T_d = \left\{ [\bm{U}_1, \ldots, \bm{U}_{d-1}, \bm{\Upsilon}_d]~:~\bm{\Upsilon}_d \in \Real^{r_{d-1} \times n_d \times r_d} \right\}
\end{gather*}
we can decompose the tangent space $T_{\bm{X}} \mathcal{M}_{\bm{r}}$ into a direct orthogonal sum
\begin{equation}\label{eq:tt:tangent_space}
    T_{\bm{X}} \mathcal{M}_{\bm{r}} = T_1 \oplus \ldots \oplus T_d.
\end{equation}
A useful fact that is derived by simple inspection is that all tensors in the tangent space $T_{\bm{X}} \mathcal{M}_{\bm{r}}$ have TT ranks that are at most $2\bm{r}$. It suffices to see that a tangent vector $\bm{Y} = \sum_{k = 1}^d \bm{Y}_k$ admits the following, non-minimal, TT representation
\begin{equation*}
    \bm{Y} = \sum_{k = 1}^d \bm{Y}_k = \left[ \begin{bmatrix} \bm{\Upsilon}_1 & \bm{U}_1 \end{bmatrix}, \begin{bmatrix} \bm{V}_2 & \bm{0} \\ \bm{\Upsilon}_2 & \bm{U}_2 \end{bmatrix}, \ldots, \begin{bmatrix} \bm{V}_{d-1} & \bm{0} \\ \bm{\Upsilon}_{d-1} & \bm{U}_{d-1} \end{bmatrix}, \begin{bmatrix} \bm{V}_d \\ \bm{\Upsilon}_d \end{bmatrix} \right],
\end{equation*}
we use block notation for the TT cores
\begin{equation*}
    \begin{bmatrix} \bm{\Upsilon}_1 & \bm{U}_1 \end{bmatrix} \in \Real^{r_0 \times n_1 \times 2r_1}, \quad \begin{bmatrix} \bm{V}_k & \bm{0} \\ \bm{\Upsilon}_k & \bm{U}_k \end{bmatrix} \in \Real^{2r_{k-1} \times n_k \times 2r_k}, \quad \begin{bmatrix} \bm{V}_d \\ \bm{\Upsilon}_d \end{bmatrix} \in \Real^{2 r_{d-1} \times n_d \times r_d}.
\end{equation*}

The formula for the orthogonal projection onto the tangent space $T_{\bm{X}} \mathcal{M}_{\bm{r}}$ was derived in \cite{LubichEtAlTime2015}. To introduce it, we need to define the tensorization operation that reverts unfoldings to tensors
\begin{equation*}
    \bm{X} = \text{ten}_k( X^{\langle k \rangle} ).
\end{equation*}
Consider the interface matrices $X_{\leq k}$ and $X_{\geq k + 1}$ for $k \in [d-1]$. Let 
\begin{equation*}
    P_{\leq k} = U_{\leq k} U_{\leq k}^T \in \Real^{(n_1 \ldots n_k) \times (n_1 \ldots n_k)}, \quad P_{\geq k+1} = V_{\geq k+1} V_{\geq k+1}^T \in \Real^{(n_{k+1} \ldots n_d) \times (n_{k+1} \ldots n_d)}
\end{equation*}
be the orthogonal projection onto their column spans. Owing to \eqref{eq:tt:interface_rec}, we can write them down recursively as 
\begin{equation*}\label{eq:tt:proj_recursive}
\begin{split}
    U_{\leq 1} &= U_1^L, \quad U_{\leq k} = (I_{n_k} \otimes U_{\leq k-1})U_{k}^L \in \Real^{(n_1 \ldots n_k) \times r_k},\\
    V_{\geq d} &= (V_d^R)^T, \quad V_{\geq k+1} = (V_{\geq k+2} \otimes I_{n_{k+1}}) (V_{k+1}^R)^T \in \Real^{(n_{k+1} \ldots n_d) \times r_k},
\end{split}
\end{equation*}
The orthogonal projection operator onto the tangent space $\mathcal{P}_{T_{\bm{X}} \mathcal{M}_{\bm{r}}} : \Real^{n_1 \times \ldots \times n_d} \to T_{\bm{X}} \mathcal{M}_{\bm{r}}$ is then given by
\begin{equation}\label{eq:tt:proj_tangent}
    \mathcal{P}_{T_{\bm{X}} \mathcal{M}_{\bm{r}}} = \sum_{k = 1}^{d - 1} (\mathcal{P}_{\leq k-1} - \mathcal{P}_{\leq k}) \mathcal{P}_{\geq k+1} + \mathcal{P}_{\leq d-1},
\end{equation}
where 
\begin{equation*}
    \mathcal{P}_{\leq k} : \bm{Z} \mapsto \text{ten}_{k} (P_{\leq k} Z^{\langle k \rangle}), \quad \mathcal{P}_{\geq k+1} : \bm{Z} \mapsto \text{ten}_{k} (Z^{\langle k \rangle}P_{\geq k+1}), \quad \mathcal{P}_{\leq 0} = \Id.
\end{equation*}

Let us try to better understand the roles played by each individual projection operator $\mathcal{P}_{\leq k}$ and $\mathcal{P}_{\geq k+1}$  \eqref{eq:tt:proj_tangent}. Denote by
\begin{equation*}
    \bm{X} = [\bm{U}_1, \ldots, \bm{U}_{d-1}, \bm{G}_d] = [\bm{G}_1, \bm{V}_2, \ldots, \bm{V}_d].
\end{equation*}
are minimal left- and right-orthogonal TT representations of $\bm{X}$. Consider a tensor $\bm{Z}$ of TT rank $\bm{r'}$ with minimal TT representation $\bm{Z} = [\bm{C}_1, \ldots, \bm{C}_d]$. The projection $\mathcal{P}_{\leq k}$ onto the column span of the left interface matrix results in a tensor with a non-minimal TT representation
\begin{equation*}
    \mathcal{P}_{\leq k} \bm{Z} = [\bm{U}_1, \ldots, \bm{U}_{k-1}, \overline{\bm{U}}_k, \bm{C}_{k+1}, \ldots, \bm{C}_d]
\end{equation*}
by replacing the $k-1$ leftmost TT cores of $\bm{Z}$ with the left-orthogonal TT cores of $\bm{X}$, keeping the $d-k$ rightmost TT cores of $\bm{Z}$, and computing a new TT core $\overline{\bm{U}}_k$ such that $\overline{U}_k^L = U_k^L W_k$ for a square matrix $W_k \in \Real^{r_k \times r'_k}$. In the same vein, $\mathcal{P}_{\geq k+1}$ produces
\begin{equation*}
    \mathcal{P}_{\geq k+1} \bm{Z} = [\bm{C}_1, \ldots, \bm{C}_{k}, \overline{\bm{V}}_{k+1}, \bm{V}_{k+2}, \ldots, \bm{V}_d]
\end{equation*}
with $\overline{V}_{k+1}^R = H_{k+1} V_{k+1}^R$, $H_{k+1} \in \Real^{r'_k \times r_k}$. It is important to note that $W_k$ and $H_{k+1}$ can be taken out of $\overline{U}_k$ and $\overline{V}_{k+1}$ and multiplied onto $\bm{C}_{k+1}$ and $\bm{C}_k$, respectively, instead:
\begin{gather*}
    \mathcal{P}_{\leq k} \bm{Z} = [\bm{U}_1, \ldots, \bm{U}_{k-1}, \bm{U}_k, \overline{\bm{C}}_{k+1}, \ldots, \bm{C}_d], \quad \overline{C}_{k+1}^R = W_k C_{k+1}^R,\\
    \mathcal{P}_{\geq k+1} \bm{Z} = [\bm{C}_1, \ldots, \overline{\bm{C}}_{k}, \bm{V}_{k+1}, \bm{V}_{k+2}, \ldots, \bm{V}_d], \quad \overline{C}_{k}^L = C_k^L H_{k+1} .
\end{gather*}
We can also deduce from these formulations that $\mathcal{P}_{\leq j}$ and $\mathcal{P}_{\geq k}$ commute when $j < k$ (even when they are connected with different tangent spaces).

Going further, we see that $\mathcal{P}_{\leq k-1} - \mathcal{P}_{\leq k}$ is a projection operator as well. Indeed, it multiplies $Z^{\langle k \rangle}$ by an orthogonal projection on the left:
\begin{align*}
    (\mathcal{P}_{\leq k - 1} - \mathcal{P}_{\leq k})\bm{Z} &= \mathrm{ten}_k \left( [(I_{n_k} \otimes P_{\leq k - 1}) - P_{\leq k}] Z^{\langle k \rangle} \right) \\
    &= \mathrm{ten}_k \left( (I_{n_k} \otimes U_{\leq k - 1})(I_{n_k r_{k-1}} - U_k^{\mathrm{L}} (U_k^{\mathrm{L}})^T) (I_{n_k} \otimes U_{\leq k - 1}^T) Z^{\langle k \rangle} \right).
\end{align*}
It is clear that for $k \in [d-1]$ every $\mathcal{P}_{\leq k-1} - \mathcal{P}_{\leq k}$ acts by imposing the orthogonal gauge condition onto the $k$-th TT core:
\begin{equation*}
    (\mathcal{P}_{\leq k-1} - \mathcal{P}_{\leq k}) \bm{Z} = [\bm{U}_1, \ldots, \bm{U}_{k-1}, \bm{\Upsilon}_k, \bm{C}_{k+1}, \ldots, \bm{C}_d], \quad (U_k^L)^T \Upsilon_k^L = 0.
\end{equation*}
And by analogy for $k \in [d-1]$ (denote $\mathcal{P}_{\geq d+1} = \text{Id}$) we have
\begin{equation*}
    (\mathcal{P}_{\geq k+2} - \mathcal{P}_{\geq k+1}) \bm{Z} = [\bm{C}_1, \ldots, \bm{C}_{k}, \bm{\Xi}_{k+1}, \bm{V}_{k+2}, \ldots, \bm{V}_d], \quad \Xi_{k+1}^R (V_{k+1}^R)^T = 0.
\end{equation*}
We can now align the decomposition of the tangent space \eqref{eq:tt:tangent_space} with the definition of the orthogonal projection operator \eqref{eq:tt:proj_tangent} since
\begin{align*}
    (\mathcal{P}_{\leq k-1} - \mathcal{P}_{\leq k}) \mathcal{P}_{\geq k+1} &: \Real^{n_1 \times \ldots \times n_d} \to T_k, \quad k \in [d-1],\\
    \mathcal{P}_{\leq d-1} &: \Real^{n_1 \times \ldots \times n_d} \to T_d.
\end{align*}
The complementary orthogonal projection operator admits a simple expression too
\begin{equation*}
        \Id - \mathcal{P} = \sum_{k = 1}^{d - 1} (\mathcal{P}_{\leq k - 1} - \mathcal{P}_{\leq k})(\Id - \mathcal{P}_{\geq k + 1}),
\end{equation*}
where we can represent each $\Id - \mathcal{P}_{\geq k + 1}$ as a sum of projections that we already understand:
\begin{equation*}
    \Id - \mathcal{P}_{\geq k + 1} = \sum_{j = k}^{d-1} (\mathcal{P}_{\geq j + 2} - \mathcal{P}_{\geq j + 1}).
\end{equation*}
\begin{proof}[Now, we are in position to prove Lemma~\ref{tt:lemma:curvature_bound}]
At first, let us show that 
\begin{equation*}
        \| P_{\leq k} - \tilde{P}_{\leq k} \| \leq \frac{\| \bm{X} - \bm{\tilde{X}} \|_F}{\sigma_{\mathrm{min}}(X^{\langle k \rangle})}, \quad \| P_{\geq k+1} - \tilde{P}_{\geq k+1} \| \leq \frac{\| \bm{X} - \bm{\tilde{X}} \|_F}{\sigma_{\min}(X^{\langle k \rangle})},
\end{equation*}
Let $X^{\langle k \rangle} = U \Sigma V^T$ be the truncated SVD of rank $r_k$. Then we have
\begin{align*}
    \| P_{\leq k} - \tilde{P}_{\leq k} \| &= \| (I - \tilde{P}_{\leq k}) P_{\leq k} \| = \| (I - \tilde{P}_{\leq k}) U U^T \| 
    = \| (I - \tilde{P}_{\leq k}) X^{\langle k \rangle} V \Sigma^{-1} U^T \| \\
    &= \| (I - \tilde{P}_{\leq k}) (X^{\langle k \rangle} - \tilde{X}^{\langle k \rangle}) V \Sigma^{-1} U^T \| \\
    &\leq \| I - \tilde{P}_{\leq k} \| \| X^{\langle k \rangle} - \tilde{X}^{\langle k \rangle} \|  \|V\| \|\Sigma^{-1} \| \|U^T\| \\
    &= \|X^{\langle k \rangle} - \tilde{X}^{\langle k \rangle} \| / \sigma_{\min} (X^{\langle k \rangle})
    \leq \| \bm{X} - \bm{\tilde{X}} \|_F / \sigma_{\min}(X^{\langle k \rangle}).
\end{align*}
An analogous argument works for the right interface matrix.
    
For brevity, denote $\mathcal{P}_{T_{\bm{X}} \mathcal{M}_{\bm{r}}}$ as $\mathcal{P}$ and $\mathcal{P}_{T_{\bm{\tilde{X}}} \mathcal{M}_{\bm{r}}}$ as $\tilde{\mathcal{P}}$. Using the decomposition of $\Id - \tilde{\mathcal{P}}$ we prove the first part of the Lemma:
\begin{align*}
    \| (\Id - \tilde{\mathcal{P}}) \bm{X} \|_F &= \Bigg\| \sum_{k = 1}^{d - 1} (\tilde{\mathcal{P}}_{\leq k - 1} - \tilde{\mathcal{P}}_{\leq k})(\Id - \tilde{\mathcal{P}}_{\geq k + 1}) \bm{X} \Bigg\|_F \\
    &\leq \sum_{k = 1}^{d - 1} \left\| (\tilde{\mathcal{P}}_{\leq k - 1} - \tilde{\mathcal{P}}_{\leq k})(\Id - \tilde{\mathcal{P}}_{\geq k + 1}) \bm{X} \right\|_F \\
    &= \sum_{k = 1}^{d - 1} \left\| (\tilde{\mathcal{P}}_{\leq k - 1} - \tilde{\mathcal{P}}_{\leq k})(\mathcal{P}_{\geq k + 1} - \tilde{\mathcal{P}}_{\geq k + 1}) \bm{X} \right\|_F \\
    &= \sum_{k = 1}^{d - 1} \left\| (\mathcal{P}_{\geq k + 1} - \tilde{\mathcal{P}}_{\geq k + 1})(\tilde{\mathcal{P}}_{\leq k - 1} - \tilde{\mathcal{P}}_{\leq k}) \bm{X} \right\|_F \\
    &\leq \sum_{k = 1}^{d - 1} \left\| \mathcal{P}_{\geq k + 1} - \tilde{\mathcal{P}}_{\geq k + 1}\right\| \left\|(\tilde{\mathcal{P}}_{\leq k - 1} - \tilde{\mathcal{P}}_{\leq k}) \bm{X} \right\|_F \\
    &= \sum_{k = 1}^{d - 1} \left\| \mathcal{P}_{\geq k + 1} - \tilde{\mathcal{P}}_{\geq k + 1}\right\| \left\|(\tilde{\mathcal{P}}_{\leq k - 1} - \tilde{\mathcal{P}}_{\leq k}) (\bm{X} - \bm{\tilde{X}}) \right\|_F \\
    &\leq \sum_{k = 1}^{d - 1} \left\| \mathcal{P}_{\geq k + 1} - \tilde{\mathcal{P}}_{\geq k + 1}\right\| \left\|\tilde{\mathcal{P}}_{\leq k - 1} - \tilde{\mathcal{P}}_{\leq k}\right\| \|\bm{X} - \bm{\tilde{X}} \|_F \\
    &= \sum_{k = 1}^{d - 1} \left\| \mathcal{P}_{\geq k + 1} - \tilde{\mathcal{P}}_{\geq k + 1}\right\| \|\bm{X} - \bm{\tilde{X}} \|_F \\
    &\leq \| \bm{X} - \bm{\tilde{X}} \|_F^2 \sum_{k=1}^{d-1} \frac{1}{\sigma_{\min}(X^{\langle k \rangle})}.
\end{align*}
Above, we also use the identity $\bm{X} = \mathcal{P}_{\geq k+1} \bm{X}$; the commutativity of $\mathcal{\tilde{P}}_{\leq k}$ and $\mathcal{P}_{\geq k+1}$; the identity $\mathcal{\tilde{P}}_{\leq k-1} \bm{\tilde{X}} = \mathcal{\tilde{P}}_{\leq k} \bm{\tilde{X}}$; and the fact that $\mathcal{\tilde{P}}_{\leq k-1} - \mathcal{\tilde{P}}_{\leq k}$ is a projector. Straightforward calculation shows that
\begin{equation*}
    \mathcal{P} - \tilde{\mathcal{P}} = \sum_{k = 1}^{d - 1} \left[ (\mathcal{P}_{\leq k} - \mathcal{\tilde{P}}_{\leq k}) (\mathcal{P}_{\geq k + 2} - \mathcal{P}_{\geq k + 1}) + (\mathcal{\tilde{P}}_{\leq k - 1} - \mathcal{\tilde{P}}_{\leq k}) (\mathcal{P}_{\geq k + 1} - \mathcal{\tilde{P}}_{\geq k + 1}) \right].
\end{equation*}
Then the second assertion follows from
\begin{align*}
    \| \mathcal{P} - \tilde{\mathcal{P}} \| &\leq \sum_{k = 1}^{d - 1} \left[ \|\mathcal{P}_{\leq k} - \mathcal{\tilde{P}}_{\leq k}\| \|\mathcal{P}_{\geq k + 2} - \mathcal{P}_{\geq k + 1}\| + \|\mathcal{\tilde{P}}_{\leq k - 1} - \mathcal{\tilde{P}}_{\leq k}\| \|\mathcal{P}_{\geq k + 1} - \mathcal{\tilde{P}}_{\geq k + 1}\| \right] \\
    &= \sum_{k = 1}^{d - 1} \left[ \|\mathcal{P}_{\leq k} - \mathcal{\tilde{P}}_{\leq k}\| + \|\mathcal{P}_{\geq k + 1} - \mathcal{\tilde{P}}_{\geq k + 1}\| \right]
    \leq 2 \| \bm{X} - \bm{\tilde{X}} \|_F \sum_{k = 1}^{d - 1} \frac{1}{\sigma_{\min}\left( X^{\langle k \rangle} \right)}. \qedhere
\end{align*}
\end{proof}

\subsection{Local convergence of Riemannian gradient descent: tensor train recovery}
To prove Theorem~\ref{recov:theorem:convergence}, we need to establish a technical Lemma~\ref{recov:lemma:rip_rip}. It shows that the standard RIP \eqref{recov:eq:rip_vanilla} implies that the RIP \eqref{complt:eq:rip} holds globally on $\mathcal{M}_{\bm{r}}$, i.e. for all of its tangent spaces.
\begin{lemma}
\label{recov:lemma:rip_rip}
Let the linear operator $\mathcal{R}$ satisfy the standard RIP \eqref{recov:eq:rip_vanilla} of order $2\bm{r}$ with a RIP constant $0 < \delta_{2\bm{r}} < 1$. Then for an arbitrary tensor $\bm{X} \in \mathcal{M}_{\bm{r}}$ of TT rank $\bm{r}$ the following RIP holds with the same constant:
\begin{equation*}
    \| \mathcal{P}_{T_{\bm{X}} \mathcal{M}_{\bm{r}}} - \mathcal{P}_{T_{\bm{X}} \mathcal{M}_{\bm{r}}} \mathcal{R}^* \mathcal{R} \mathcal{P}_{T_{\bm{X}} \mathcal{M}_{\bm{r}}} \| < \delta_{2\bm{r}}.
\end{equation*}
\end{lemma}
\begin{proof}
Observe that $\mathcal{P}_{T_{\bm{X}} \mathcal{M}_{\bm{r}}} - \mathcal{P}_{T_{\bm{X}} \mathcal{M}_{\bm{r}}} \mathcal{R}^* \mathcal{R} \mathcal{P}_{T_{\bm{X}} \mathcal{M}_{\bm{r}}}$ is a self-adjoint operator so its norm can be characterized as
\begin{equation*}
    \| \mathcal{P}_{T_{\bm{X}} \mathcal{M}_{\bm{r}}} - \mathcal{P}_{T_{\bm{X}} \mathcal{M}_{\bm{r}}} \mathcal{R}^* \mathcal{R} \mathcal{P}_{T_{\bm{X}} \mathcal{M}_{\bm{r}}} \| = \max_{\bm{Z} : \| \bm{Z} \|_F = 1} \langle (\mathcal{P}_{T_{\bm{X}} \mathcal{M}_{\bm{r}}} - \mathcal{P}_{T_{\bm{X}} \mathcal{M}_{\bm{r}}} \mathcal{R}^* \mathcal{R} \mathcal{P}_{T_{\bm{X}} \mathcal{M}_{\bm{r}}}) \bm{Z}, \bm{Z} \rangle_F.
\end{equation*}
It follows that
\begin{align*}
    \| \mathcal{P}_{T_{\bm{X}} \mathcal{M}_{\bm{r}}} - \mathcal{P}_{T_{\bm{X}} \mathcal{M}_{\bm{r}}} \mathcal{R}^* \mathcal{R} \mathcal{P}_{T_{\bm{X}} \mathcal{M}_{\bm{r}}} \| &= \max_{\bm{Z} : \| \bm{Z} \|_F = 1} \left( \| \mathcal{P}_{T_{\bm{X}} \mathcal{M}_{\bm{r}}} \bm{Z} \|_F^2 - \| \mathcal{R} \mathcal{P}_{T_{\bm{X}} \mathcal{M}_{\bm{r}}} \bm{Z} \|_F^2 \right) \\
    &\leq \max_{\bm{Z} : \| \bm{Z} \|_F = 1} \left( \delta_{2\bm{r}} \| \mathcal{P}_{T_{\bm{X}} \mathcal{M}_{\bm{r}}} \bm{Z} \|_F^2 \right) \leq \delta_{2\bm{r}}
\end{align*}
because the elements of every tangent space to $\mathcal{M}_{\bm{r}}$ have ranks equal to at most $2\bm{r}$.
\end{proof}

In the proof of Theorem~\ref{recov:theorem:convergence}, we are only going to use the result of Lemma~\ref{recov:lemma:rip_rip} and not the standard RIP \eqref{recov:eq:rip_vanilla} itself. This explains why the proof can be adapted to the TT completion case.
\begin{proof}[Proof of Theorem~\ref{recov:theorem:convergence}]
The new iterate is given by \eqref{recov:eq:rgd} so by using the quasi-optimality of TT-SVD projection we get
\begin{align*}
    \| \bm{X}_{t+1} - \bm{A} \|_F &= \| \text{TT-SVD}_{\bm{r}}(\bm{X}_t - \alpha_t \bm{Y}_t) - \bm{A} \|_F \\
    &\leq \| \text{TT-SVD}_{\bm{r}}(\bm{X}_t - \alpha_t \bm{Y}_t) - (\bm{X}_t - \alpha_t \bm{Y}_t) \|_F + \| (\bm{X}_t - \alpha_t \bm{Y}_t) - \bm{A} \|_F \\
    &\leq \sqrt{d - 1} \| \text{opt}_{\bm{r}}(\bm{X}_t - \alpha_t \bm{Y}_t) - (\bm{X}_t - \alpha_t \bm{Y}_t) \|_F + \| (\bm{X}_t - \alpha_t \bm{Y}_t) - \bm{A} \|_F \\
    &\leq (1 + \sqrt{d - 1}) \| (\bm{X}_t - \alpha_t \bm{Y}_t) - \bm{A} \|_F.
\end{align*}
We then separate this Frobenius norm into a sum of several components that we will bound one by one
\begin{align*}
    \| (\bm{X}_t - \alpha_t \bm{Y}_t) - \bm{A} \|_F  &= \| \bm{X}_t - \alpha_t \mathcal{P}_{\bm{X}_t} \mathcal{R}^* \mathcal{R} (\bm{X}_t - \bm{A}) - \bm{A} \|_F 
    = \| (\Id - \alpha_t \mathcal{P}_{\bm{X}_t} \mathcal{R}^* \mathcal{R}) (\bm{X}_t - \bm{A}) \|_F \\
    &\leq \| (\Id - \mathcal{P}_{\bm{X}_t}) (\bm{X}_t - \bm{A}) \|_F + \| (\mathcal{P}_{\bm{X}_t} - \mathcal{P}_{\bm{X}_t} \mathcal{R}^* \mathcal{R} \mathcal{P}_{\bm{X}_t}) (\bm{X}_t - \bm{A}) \|_F \\
    &+ |1 - \alpha_t| \| \mathcal{P}_{\bm{X}_t} \mathcal{R}^* \mathcal{R} \mathcal{P}_{\bm{X}_t} (\bm{X}_t - \bm{A}) \|_F + |\alpha_t| \| \mathcal{P}_{\bm{X}_t} \mathcal{R}^* \mathcal{R} (\Id - \mathcal{P}_{\bm{X}_t}) (\bm{X}_t - \bm{A}) \|_F.
\end{align*}
For the first term we use the curvature bound Lemma \ref{tt:lemma:curvature_bound} to get
\begin{equation*}
    \| (\Id - \mathcal{P}_{\bm{X}_t}) (\bm{X}_t - \bm{A}) \|_F \leq (d-1) \frac{\| \bm{X}_t - \bm{A} \|_F^2}{\sigma_{\min}(\bm{A})}.
\end{equation*}
The bound for the second term follows from Lemma~\ref{recov:lemma:rip_rip}:
\begin{equation*}
    \| (\mathcal{P}_{\bm{X}_t} - \mathcal{P}_{\bm{X}_t} \mathcal{R}^* \mathcal{R} \mathcal{P}_{\bm{X}_t}) (\bm{X}_t - \bm{A}) \|_F \leq \delta_{2\bm{r}} \| \bm{X}_t - \bm{A} \|_F.
\end{equation*}
To estimate the third term, we note that the step size $\alpha_t = \| \bm{Y}_t \|_F^2 / \| \mathcal{R} \bm{Y}_t \|_F^2$ is close to one. Indeed, $\bm{Y}_t$ has TT ranks at most $2\bm{r}$ since it belongs to the tangent space and so
\begin{equation*}
    \frac{1}{1 + \delta_{2\bm{r}}} \leq \alpha_t \leq \frac{1}{1 - \delta_{2\bm{r}}}.
\end{equation*}
We then use the variational characterization of the Frobenius norm
\begin{align*}
    \| \mathcal{P}_{\bm{X}_t} \mathcal{R}^* \mathcal{R} \mathcal{P}_{\bm{X}_t} (\bm{X}_t - \bm{A}) \|_F &= \max_{\bm{Z} : \| \bm{Z} \|_F = 1} \langle \mathcal{P}_{\bm{X}_t} \mathcal{R}^* \mathcal{R} \mathcal{P}_{\bm{X}_t} (\bm{X}_t - \bm{A}), \bm{Z} \rangle_F\\
    &= \max_{\bm{Z} : \| \bm{Z} \|_F = 1} \langle \mathcal{R} \mathcal{P}_{\bm{X}_t} (\bm{X}_t - \bm{A}), \mathcal{R} \mathcal{P}_{\bm{X}_t} \bm{Z} \rangle_F \\
    &\leq \max_{\bm{Z} : \| \bm{Z} \|_F = 1} \| \mathcal{R} \mathcal{P}_{\bm{X}_t} (\bm{X}_t - \bm{A}) \|_F  \| \mathcal{R} \mathcal{P}_{\bm{X}_t} \bm{Z} \|_F\\
    &\leq \max_{\bm{Z} : \| \bm{Z} \|_F = 1} (1 + \delta_{2\bm{r}}) \| \mathcal{P}_{\bm{X}_t} (\bm{X}_t - \bm{A}) \|_F  \| \mathcal{P}_{\bm{X}_t} \bm{Z} \|_F \\
    &\leq  (1 + \delta_{2\bm{r}}) \| \bm{X}_t - \bm{A} \|_F.
\end{align*}
Thus the third term is bounded by
\begin{equation*}
    |1 - \alpha_t| \| \mathcal{P}_{\bm{X}_t} \mathcal{R}^* \mathcal{R} \mathcal{P}_{\bm{X}_t} (\bm{X}_t - \bm{A}) \|_F \leq \delta_{2\bm{r}} \frac{1 + \delta_{2\bm{r}}}{1 - \delta_{2\bm{r}}} \| \bm{X}_t - \bm{A} \|_F.
\end{equation*}
For the fourth term, we use the operator norm bound $\| \mathcal{R}^* \mathcal{R} \| \leq C$:
\begin{equation*}
    |\alpha_t| \| \mathcal{P}_{\bm{X}_t} \mathcal{R}^* \mathcal{R} (\Id - \mathcal{P}_{\bm{X}_t}) (\bm{X}_t - \bm{A}) \|_F \leq \frac{C}{1 - \delta_{2\bm{r}}} (d-1) \frac{\| \bm{X}_t - \bm{A} \|_F^2}{\sigma_{\min}(\bm{A})}.
\end{equation*}
Finally, collecting the terms, we get
\begin{equation*}
    \| \bm{X}_{t+1} - \bm{A} \|_F \leq (1 + \sqrt{d-1}) \left[ \frac{2\delta_{2\bm{r}}}{1 - \delta_{2\bm{r}}} + \left( 1 + \frac{C}{1 - \delta_{2\bm{r}}} \right) (d-1) \frac{\| \bm{X}_t - \bm{A} \|_F}{\sigma_{\min}(\bm{A})} \right] \| \bm{X}_t - \bm{A} \|_F.
\end{equation*}
If the initial condition $\bm{X}_0 \in \mathcal{M}_{\bm{r}}$ is close enough
\begin{equation*}
    (d-1) \frac{\| \bm{X}_0 - \bm{A} \|_F}{\sigma_{\min}(\bm{A})} < \frac{1}{1 + C - \delta_{2\bm{r}}} \left( \frac{1 - \delta_{2\bm{r}}}{1 + \sqrt{d-1}} - 2\delta_{2\bm{r}} \right),
\end{equation*}
the rate $\beta_0$ becomes smaller than one and as a consequence $\beta_t < \beta_0 < 1$.

To prove the final assertion we note that the TT rank of $(\Id - \mathcal{P}_{\bm{X}_t}) (\bm{X}_t - \bm{A})$ is at most $3\bm{r}$ and so the standard RIP can be used to estimate the fourth term:
\begin{equation*}
    |\alpha_t| \| \mathcal{P}_{\bm{X}_t} \mathcal{R}^* \mathcal{R} (\Id - \mathcal{P}_{\bm{X}_t}) (\bm{X}_t - \bm{A}) \|_F \leq \frac{1 + \delta_{3\bm{r}}}{1 - \delta_{2\bm{r}}} (d-1) \frac{\| \bm{X}_t - \bm{A} \|_F^2}{\sigma_{\min}(\bm{A})},
\end{equation*}
where we used the variational form of the Frobenius norm and the fact that $\delta_{2\bm{r}} \leq \delta_{3\bm{r}}$.
\end{proof}

\subsection{Local convergence of Riemannian gradient descent: tensor train completion}
For TT completion, we need an analog of Lemma~\ref{recov:lemma:rip_rip}. While the standard RIP \eqref{recov:eq:rip_vanilla} guarantees that the RIP \eqref{complt:eq:rip} holds globally, Lemma~\ref{complt:lemma:rip_rip} shows that if the RIP \eqref{complt:eq:rip} holds on a particular tangent space, so does it locally on the neighboring tangent spaces, though with a degrading constant.
\begin{lemma}
\label{complt:lemma:rip_rip}
Let $\bm{A} \in \mathcal{M}_{\bm{r}}$ be a tensor of TT rank $\bm{r}$ and suppose that $\mathcal{R}_{\Omega}$ satisfies the RIP \eqref{complt:eq:rip} and is bounded
\begin{equation*}
    \| \mathcal{P}_{T_{\bm{A}} \mathcal{M}_{\bm{r}}} - \rho^{-1} \mathcal{P}_{T_{\bm{A}} \mathcal{M}_{\bm{r}}} \mathcal{R}_{\Omega} \mathcal{P}_{T_{\bm{A}} \mathcal{M}_{\bm{r}}} \| < \varepsilon, \quad \| \mathcal{R}_{\Omega} \| \leq C.
\end{equation*}
Then for every tensor $\bm{X} \in \mathcal{M}_{\bm{r}}$ with a sufficiently close tangent space $\| \mathcal{P}_{T_{\bm{A}} \mathcal{M}_{\bm{r}}} - \mathcal{P}_{T_{\bm{X}} \mathcal{M}_{\bm{r}}} \| < \delta$, the sampling operator $\mathcal{R}_{\Omega}$ satisfies the RIP \eqref{complt:eq:rip} on it as well
\begin{equation*}
    \| \mathcal{P}_{T_{\bm{X}} \mathcal{M}_{\bm{r}}} - \rho^{-1} \mathcal{P}_{T_{\bm{X}} \mathcal{M}_{\bm{r}}} \mathcal{R}_{\Omega} \mathcal{P}_{T_{\bm{X}} \mathcal{M}_{\bm{r}}} \| < E(\delta) \equiv \varepsilon + \delta \left( 1 + 2 C \rho^{-1} \right).
\end{equation*}
\end{lemma}
\begin{proof}
Denote by $\mathcal{P}_{\bm{A}}$ the projection $\mathcal{P}_{T_{\bm{A}} \mathcal{M}_{\bm{r}}}$ and similarly for $\mathcal{P}_{\bm{X}}$. Then
\begin{align*}
    \| \mathcal{P}_{\bm{X}} - \rho^{-1} \mathcal{P}_{\bm{X}} \mathcal{R}_\Omega \mathcal{P}_{\bm{X}} \| &\leq \| \mathcal{P}_{\bm{A}} - \rho^{-1} \mathcal{P}_{\bm{A}} \mathcal{R}_\Omega \mathcal{P}_{\bm{A}} \| + \| \mathcal{P}_{\bm{X}} - \mathcal{P}_{\bm{A}} \| + \rho^{-1} \| \mathcal{P}_{\bm{X}} \mathcal{R}_{\Omega} \mathcal{P}_{\bm{X}} - \mathcal{P}_{\bm{A}} \mathcal{R}_\Omega \mathcal{P}_{\bm{A}} \| \\
    &\leq \varepsilon + \delta + \rho^{-1} \| \mathcal{P}_{\bm{X}} \mathcal{R}_{\Omega} \mathcal{P}_{\bm{X}} - \mathcal{P}_{\bm{X}} \mathcal{R}_\Omega \mathcal{P}_{\bm{A}} \| + \rho^{-1} \| \mathcal{P}_{\bm{X}} \mathcal{R}_{\Omega} \mathcal{P}_{\bm{A}} - \mathcal{P}_{\bm{A}} \mathcal{R}_\Omega \mathcal{P}_{\bm{A}} \| \\ 
    &\leq \varepsilon + \delta + \rho^{-1} \| \mathcal{P}_{\bm{X}} - \mathcal{P}_{\bm{A}} \| (\| \mathcal{P}_{\bm{X}} \mathcal{R}_{\Omega} \| + \| \mathcal{R}_\Omega \mathcal{P}_{\bm{A}} \|) \\
    &\leq \varepsilon + \delta \left( 1 + 2 C \rho^{-1} \right).
\end{align*}
A tighter bound can be derived if we estimate $\| \mathcal{R}_\Omega \mathcal{P}_{\bm{A}} \|$ with more care using the RIP, see \cite{WeiEtAlGuarantees2016}.
\end{proof}
Knowing how the RIP \eqref{complt:eq:rip} behaves in a neighborhood of $\bm{A}$, we can adapt the proof of Theorem~\ref{recov:theorem:convergence}. Note that neither Lemma~\ref{complt:lemma:rip_rip} nor Theorem~\ref{complt:theorem:convergence} exploits the actual nature of $\mathcal{R}_{\Omega}$, so their proofs apply to any other measurement operator with the RIP \eqref{complt:eq:rip}.
\begin{proof}[Proof of Theorem~\ref{complt:theorem:convergence}]
We basically repeat the proof of Theorem \ref{recov:theorem:convergence} with certain modifications related to the RIP \eqref{complt:eq:rip}. We immediately get that
\begin{equation*}
    \| \bm{X}_{t+1} - \bm{A} \|_F \leq (1 + \sqrt{d-1}) \| (\bm{X}_t - \alpha_t \bm{Y}_t) - \bm{A} \|_F
\end{equation*}
and
\begin{align*}
    \| (\bm{X}_t - \alpha_t \bm{Y}_t) - \bm{A} \|_F &\leq \| (\Id - \mathcal{P}_{\bm{X}_t}) (\bm{X}_t - \bm{A}) \|_F + \| (\mathcal{P}_{\bm{X}_t} - \rho^{-1} \mathcal{P}_{\bm{X}_t} \mathcal{R}_\Omega \mathcal{P}_{\bm{X}_t}) (\bm{X}_t - \bm{A}) \|_F \\
    &+ |\rho^{-1} - \alpha_t| \| \mathcal{P}_{\bm{X}_t} \mathcal{R}_\Omega \mathcal{P}_{\bm{X}_t} (\bm{X}_t - \bm{A}) \|_F + |\alpha_t| \| \mathcal{P}_{\bm{X}_t} \mathcal{R}_\Omega (\Id - \mathcal{P}_{\bm{X}_t}) (\bm{X}_t - \bm{A}) \|_F.
\end{align*}
Each term is then estimated using Lemmas~\ref{tt:lemma:curvature_bound} and \ref{complt:lemma:rip_rip}. We only need to bound $\alpha_t = \| \bm{Y}_t \|_F^2 / \langle \mathcal{R}_\Omega \bm{Y}_t, \bm{Y}_t \rangle_F$. Denote
\begin{equation*}
    \delta_t = 2 (d-1) \frac{\| \bm{X}_t - \bm{A} \|_F}{\sigma_{\min}(\bm{A})}.
\end{equation*}
The operator $\mathcal{P}_{\bm{X}_t} - \rho^{-1} \mathcal{P}_{\bm{X}_t} \mathcal{R}_\Omega \mathcal{P}_{\bm{X}_t}$ being self-adjoint, we have
\begin{equation*}
    -E(\delta_t) \langle \bm{Y}_t, \bm{Y}_t \rangle_F < \langle (\rho^{-1} \mathcal{P}_t \mathcal{R}_{\Omega} \mathcal{P}_t - \mathcal{P}_t) \bm{Y}_t, \bm{Y}_t \rangle_F < E(\delta_t) \langle \bm{Y}_t, \bm{Y}_t \rangle_F.
\end{equation*}
As a consequence,
\begin{equation*}
    \frac{\rho^{-1}}{1 + E(\delta_t)} \leq \alpha_t \leq \frac{\rho^{-1}}{1 - E(\delta_t)}
\end{equation*}
and the Theorem follows.
\end{proof}

\subsection{Recovery guarantees}
The interface incoherence assumption that we make in Theorem~\ref{guarantee:theorem:rip_interface} allows us to estimate the norm of the projection of a canonical basis tensor $\bm{E}_\omega \in \Real^{n_1 \times \ldots \times n_d}$ onto the tangent space $T_{\bm{A}} \mathcal{M}_{\bm{r}}$, i.e. estimate the coherence of the tangent space.
\begin{lemma}
\label{guarantee:lemma:interface_projection_bound}
Let $\bm{A} \in \mathcal{M}_{\bm{r}}$ be a tensor of TT rank $\bm{r}$ with bounded interface coherence $\mu_I(\bm{A}) \leq \mu_0$. Then for every canonical basis tensor $\bm{E}_{\omega}$, $\omega \in [n_1] \times \ldots \times [n_d]$, its projection onto the tangent space $T_{\bm{A}} \mathcal{M}_{\bm{r}}$ can be bounded from above as
\begin{equation*}
    \| \mathcal{P}_{T_{\bm{A}} \mathcal{M}_{\bm{r}}} \bm{E}_\omega \|_F^2 \leq C_0 \equiv \frac{\mu_0}{n_1 \ldots n_d}\left(n_1 r_1 + \mu_0 \sum_{k = 2}^{d-1} r_{k - 1} n_k r_k + r_{d-1}n_d \right).
\end{equation*}
\end{lemma}
\begin{proof}
Every canonical basis tensor $\bm{E}_{\omega}$ can be represented as an outer product of canonical basis vectors $\bm{E}_{\omega} = e_{i_1} \circ \ldots \circ e_{i_d}$ with $e_{i_k} \in \Real^{n_k}$. Then using the definition of the projection onto the tangent space \eqref{eq:tt:proj_tangent} we get
\begin{align*}
    \| \mathcal{P}_{T_{\bm{A}} \mathcal{M}_{\bm{r}}} \bm{E}_\omega \|_F^2 &= \sum_{k = 1}^{d - 1} \Big[ \| \mathcal{P}_{\leq k-1} \mathcal{P}_{\geq k+1} \bm{E}_\omega \|_F^2 - \| \mathcal{P}_{\leq k} \mathcal{P}_{\geq k+1} \bm{E}_\omega \|_F^2 \Big] + \| \mathcal{P}_{\leq d-1} \bm{E}_\omega \|_F^2 \\
    &\leq \| \mathcal{P}_{\geq 2} \bm{E}_\omega \|_F^2 + \sum_{k = 2}^{d-1} \| \mathcal{P}_{\leq k-1} \mathcal{P}_{\geq k+1} \bm{E}_\omega \|_F^2 + \| \mathcal{P}_{\leq d-1} \bm{E}_\omega \|_F^2.
\end{align*}
The first and last terms are bounded directly using the interface incoherence property because
\begin{equation*}
    \| \mathcal{P}_{\geq 2} \bm{E}_\omega \|_F^2 = \| e_{i_1} (e_{i_2} \otimes \ldots \otimes e_{i_d})^T P_{\geq 2} \|_F^2 = \| P_{\geq 2} (e_{i_2} \otimes \ldots \otimes e_{i_d}) \|_2^2 \leq \frac{r_1}{n_2 \ldots n_d} \mu_0
\end{equation*}
and
\begin{equation*}
    \| \mathcal{P}_{\leq d-1} \bm{E}_\omega \|_F^2 = \| P_{\leq d-1} (e_{i_{d-1}} \otimes \ldots \otimes e_{i_1}) e_{i_d}^T \|_F^2 = \| P_{\leq d-1} (e_{i_{d-1}} \otimes \ldots \otimes e_{i_1}) \|_2^2 \leq \frac{r_{d-1}}{n_1 \ldots n_{d-1}} \mu_0.
\end{equation*}
We then estimate every summand $\| \mathcal{P}_{\leq k-1} \mathcal{P}_{\geq k+1} \bm{E}_\omega \|_F^2$ as follows
\begin{align*}
    \| \mathcal{P}_{\leq k-1} \mathcal{P}_{\geq k+1} \bm{E}_\omega \|_F^2 &= \| P_{\leq k-1}(e_{i_{k - 1}} \otimes \ldots \otimes e_{i_1}) \circ~e_{i_k} \circ P_{\geq k + 1} (e_{i_{k+1}} \otimes \ldots \otimes e_{i_d}) \|_F^2 \\
    &= \| P_{\leq k-1}(e_{i_{k - 1}} \otimes \ldots \otimes e_{i_1}) \|_2^2 \cdot \| P_{\geq k + 1} (e_{i_{k+1}} \otimes \ldots \otimes e_{i_d}) \|_2^2 \\
    &\leq \frac{r_{k - 1}}{n_1 \ldots n_{k-1}} \mu_0 \frac{r_k}{n_{k+1} \ldots n_d} \mu_0.
\end{align*}
It remains to add the estimates together.
\end{proof}

The main probabilistic tool we need in order to prove Theorem \ref{guarantee:theorem:rip_interface} is the noncommutative Bernstein inequality (Theorem~\ref{guarantee:theorem:bernstein}), which is used in analyzing large deviation bounds.
\begin{theorem}[\cite{RechtSimpler2011}, Theorem 4]
\label{guarantee:theorem:bernstein}
Let $X_1, \ldots, X_K \in \Real^{s_1 \times s_2}$ be independent zero-mean random matrices. Suppose that
\begin{equation*}
    \sigma^2_k = \max \left( \Big\| \E\big[X_k X_k^T\big] \Big\|_2, \Big\| \E\left[X_k^T X_k\right] \Big\|_2 \right)
\end{equation*}
and $\| X_k \|_2 \leq R$ almost surely for every $k$. Then for any $\tau > 0$,
\begin{equation*}
    \Prob\left\{  \bigg\| \sum_{k = 1}^K X_k \bigg\|_2 > \tau \right\} \leq (s_1 + s_2) \exp\left( \frac{-\tau^2/2}{\sum_{k = 1}^K \sigma^2_k + R \tau/3} \right).
\end{equation*}
If in addition $\tau \leq \sum_{k=1}^K \sigma^2_k / R$,
\[
\Prob\left\{  \bigg\| \sum_{k = 1}^K X_k \bigg\|_2 > \tau \right\} \leq (s_1 + s_2) \exp\left( \frac{-\frac{3}{8}\tau^2}{\sum_{k = 1}^K \sigma^2_k} \right).
\]
\end{theorem}

\begin{proof}[Proof of Theorem \ref{guarantee:theorem:rip_interface}]
Since any tensor $\bm{Z}$ can be represented as a linear combination of canonical basis tensors
\begin{equation*}
    \bm{Z} = \sum_{\omega \in [n_1] \times \ldots \times [n_d]} \bm{Z}(i_1, \ldots, i_d) \bm{E}_\omega  = \sum_{\omega \in [n_1] \times \ldots \times [n_d]} \langle \bm{Z}, \bm{E}_\omega \rangle_F \bm{E}_\omega,
\end{equation*}
the application of the operator $\mathcal{P}_{\bm{A}} \mathcal{R}_{\Omega} \mathcal{P}_{\bm{A}}$---where we once again write $\mathcal{P}_{\bm{A}}$ as a shorthand for $\mathcal{P}_{T_{\bm{A}} \mathcal{M}_{\bm{r}}}$---can be computed as
\begin{equation*}
    \mathcal{P}_{\bm{A}} \mathcal{R}_{\Omega} \mathcal{P}_{\bm{A}} \bm{Z} = \mathcal{P}_{\bm{A}} \left( \sum_{\omega \in \Omega} \langle \mathcal{P}_{\bm{A}} \bm{Z}, \bm{E}_\omega \rangle_F \bm{E}_\omega \right) = \sum_{\omega \in \Omega} \langle \bm{Z}, \mathcal{P}_{\bm{A}} \bm{E}_{\omega} \rangle_F \mathcal{P}_{\bm{A}} \bm{E}_{\omega}.
\end{equation*}
Every $\omega \in \Omega$ is a uniformly distributed random variable so $\mathcal{P}_{\bm{A}} \mathcal{R}_{\Omega} \mathcal{P}_{\bm{A}}$ is a sum of $|\Omega|$  i.i.d. random operators
\begin{equation*}
    \mathcal{P}_{\bm{A}} \mathcal{R}_{\Omega} \mathcal{P}_{\bm{A}} = \sum_{\omega \in \Omega} \mathcal{S}_{\omega}, \quad 
    \mathcal{S}_\omega \bm{Z} = \langle \bm{Z}, \mathcal{P}_{\bm{A}} \bm{E}_{\omega} \rangle_F \mathcal{P}_{\bm{A}} \bm{E}_{\omega}.
\end{equation*}
The expected value of $\mathcal{S}_{\omega}$ is $\frac{1}{n_1 \ldots n_d} \mathcal{P}_{\bm{A}}$ and we can estimate the norm of the deviation as
\begin{equation*}
    \| \mathcal{S}_{\omega} - \tfrac{1}{n_1 \ldots n_d} \mathcal{P}_{\bm{A}} \| \leq \max\left( \| \mathcal{S}_{\omega} \|, \tfrac{1}{n_1 \ldots n_d} \| \mathcal{P}_{\bm{A}} \| \right) = \max \left( \| \mathcal{P}_{\bm{A}} \bm{E}_\omega \|_F^2,  \tfrac{1}{n_1 \ldots n_d} \right) = C_0.
\end{equation*}
The first inequality holds since both $\mathcal{S}_{\omega}$ and $\frac{1}{n_1 \ldots n_d} \mathcal{P}_{\bm{A}}$ are positive semidefinite. To apply the noncommutative Bernstein inequality we also need a bound for the variance of $\mathcal{S}_\omega$:
\begin{align*}
    \left\| \E \{ \mathcal{S}_{\omega} - \tfrac{1}{n_1 \ldots n_d} \mathcal{P}_{\bm{A}} \}^2 \right\| &= \left\| \E \{ \| \mathcal{P}_{\bm{A}} \bm{E}_\omega \|_F^2  \mathcal{S}_{\omega} \} - \tfrac{1}{(n_1 \ldots n_d)^2} \mathcal{P}_{\bm{A}} \right\| 
    \leq \max \left( \left\| \E \{ \| \mathcal{P}_{\bm{A}} \bm{E}_\omega \|_F^2 \mathcal{S}_{\omega} \} \right\|, \tfrac{1}{(n_1 \ldots n_d)^2} \right) \\
    &\leq \max \left( \tfrac{C_{0}}{n_1 \ldots n_d}, \tfrac{1}{(n_1 \ldots n_d)^2} \right) 
    = \frac{C_{0}}{n_1 \ldots n_d}.
\end{align*}
We then apply the second part of Theorem \ref{guarantee:theorem:bernstein} to $\mathcal{S}_{\omega} - \tfrac{1}{n_1 \ldots n_d} \mathcal{P}_{\bm{A}}$ for $\omega \in \Omega$. When $\tau / \rho = \varepsilon < 1$, we have
\begin{align*}
    \Prob \left\{ \| \mathcal{P}_{\bm{A}} - \rho^{-1} \mathcal{P}_{\bm{A}} \mathcal{R}_{\Omega} \mathcal{P}_{\bm{A}} \| > \tau / \rho = \varepsilon \right\} &\leq 2 (n_1 \ldots n_d) \exp\left( -\frac{3}{8} \frac{\tau^2}{\rho C_{0}} \right)
    \leq 2n^d \exp\left( -\frac{3}{8} \frac{\rho \varepsilon^2}{C_{0}} \right)
    \leq 2n^{d(1 - \beta)}
\end{align*}
provided that $\rho \geq \frac{8}{3} \frac{C_{0}}{\varepsilon^2} d\beta\log(n)$.
\end{proof}

Theorem \ref{guarantee:theorem:rip_core} is proved in exactly the same way as Theorem \ref{guarantee:theorem:rip_interface}, except we use a refined estimate of $\| \mathcal{P}_{T_{\bm{A}} \mathcal{M}_{\bm{r}}} \bm{E}_\omega \|_F$ that is given in the following Lemma~\ref{guarantee:lemma:core_projection_bound}.

\begin{lemma}\label{guarantee:lemma:core_projection_bound}
Let $\bm{A} \in \mathcal{M}_{\bm{r}}$ be a tensor of TT rank $\bm{r}$ with bounded core coherence $\mu_C(\bm{A}) \leq \mu_1$.Then for every canonical basis tensor $\bm{E}_{\omega}$, $\omega \in [n_1] \times \ldots \times [n_d]$, its projection onto the tangent space $T_{\bm{A}} \mathcal{M}_{\bm{r}}$ can be bounded from above as
\[
    \| \mathcal{P}_{T_{\bm{A}} \mathcal{M}_{\bm{r}}} \bm{E}_\omega \|_F^2 \leq \frac{\mu_1^{d-1}}{n_1 \ldots n_d} \sum_{k = 1}^{d} r_{k - 1} n_k r_k.
\]
\end{lemma}
\begin{proof}
The proof goes along the same line as Lemma~\ref{guarantee:lemma:interface_projection_bound} and uses Lemma~\ref{guarantee:lemma:core2interface} to obtain the bounds.
\end{proof}
\section{Discussion}\label{section:discussion}
The sample complexities that we obtained for TT completion (Theorem~\ref{guarantee:theorem:rip_core}) and TT completion with auxiliary subspace information (Theorem~\ref{side:theorem:rip}) depend on the core coherence \eqref{guarantee:eq:core_coherence} as $\mu_C(\bm{A})^{d-1}$. It is, thus, important to have a qualitative estimate of how large the core coherence can be. In the matrix case, \cite{CandesRechtExact2009a}, $\mu_C(A)$ was proved to be of order $\max(r, \log(n))$ for matrices, whose left and right singular factors are chosen uniformly at random from the set of $n \times r$ matrices with orthonormal columns $\text{St}(n, r)$. To sample such factors, one can take a random $n \times r$ matrix with i.i.d standard Gaussian entries and apply Gram-Schmidt orthogonalization \cite{EatonMultivariate1983}.

Consider now a minimal left-orthogonal TT representation of a tensor $\bm{A} = [\bm{U}_1, \ldots, \bm{U}_{d-1}, \bm{G}_d]$, whose TT cores $\bm{U}_k$ are sampled uniformly from $\text{St}(r_{k-1} n_k, r_k)$. What can be said about the distribution of their subblocks $U_k^{(i_k)}$ that are used in the definition of the left coherence of a TT core \eqref{guarantee:eq:left_coherence}? It is known that if we take a random orthogonal matrix $Q^{(n)} \in \Real^{n \times n}$, pick any of its subblocks $Q^{(n)}_{p,q} \in \Real^{p \times q}$, and let $n \to \infty$, then the matrix $\sqrt{n} Q^{(n)}_{p,q}$ converges in distribution to a matrix with i.i.d. standard Gaussian entries \cite{EatonGroup1989, DiaconisEtAlFinite1992}. As a consequence, we can, informally, treat the blocks $\sqrt{n_k} U_{k}^{(i_k)}$ as random matrices sampled from the standard Gaussian distribution. Random matrix theory provides probabilistic estimates on the spectral norm of a standard Gaussian random matrix \cite{VershyninIntroduction2011}. With probability at least $1 - 2\exp(-t^2/2)$, we have
\begin{equation*}
    \| \sqrt{n_k} U_{k}^{(i_k)} \|_2 \leq \sqrt{r_{k-1}} + \sqrt{r_k} + t.
\end{equation*}
It follows that, with high probability,
\begin{equation*}
    \tfrac{r_{k-1} n_k}{r_k} \| U_{k}^{(i_k)} \|_2^2 \leq \tfrac{r_{k-1}}{r_k} (\sqrt{r_{k-1}} + \sqrt{r_k} + t)^2
\end{equation*}
and so $\mu_C(\bm{A})$ should be of order $\max(r, \log(n))$ as well, if we set $t = c \sqrt{\log(n)}$.

The exponential dependence $\mu_C(\bm{A})^{d-1}$ originates in Lemma \ref{guarantee:lemma:core2interface}, where we bound the spectral norms of the row vectors
\begin{equation*}
    U_{1}^{(i_1)} U_{2}^{(i_2)} \ldots U_{k}^{(i_k)}
\end{equation*}
using the submultiplicative property. Assume, once again, that we can informally treat the subblocks $\sqrt{n_k} U_{k}^{(i_k)}$ as standard Gaussian random matrices. The product of a Gaussian random matrix and a Gaussian random vector has a known distribution \cite{MatteiMultiplying2017}. In our case, the first product $(U_{1}^{(i_1)} U_{2}^{(i_2)})^T \in \Real^{r_2}$ is distributed as
\begin{equation*}
    \sqrt{n_1 n_2}(U_{1}^{(i_1)} U_{2}^{(i_2)})^T \sim \sqrt{s_1(r_1)} z,
\end{equation*}
where $s_1(r_1) \sim \chi^2(r_1)$ is a chi-squared random variable with $r_1$ degrees of freedom and $z \in \Real^{r_2}$ is a standard Gaussian random vector independent of $s_1$. Multiplying further, we find that
\begin{equation*}
    \sqrt{n_1 \ldots n_k}(U_{1}^{(i_1)} U_{2}^{(i_2)} \ldots U_{k}^{(i_k)})^T \sim \sqrt{s_1(r_1) \ldots s_{k-1}(r_{k-1})} z
\end{equation*}
with a standard Gaussian random vector $z \in \Real^{r_k}$. The squared Euclidean norm of this vector is distributed as a product of $k$ independent chi-squared random variables with the number of degrees of freedom equal to the corresponding TT rank:
\begin{equation}\label{eq:recov:chichi}
    (n_1 \ldots n_k) \left\| U_{1}^{(i_1)} U_{2}^{(i_2)} \ldots U_{k}^{(i_k)} \right\|_2^2 \sim s_1(r_1) \ldots s_k(r_{k}).
\end{equation}
Its expectation is a good reference value to compare $\mu(A_{\leq k})$ against:
\begin{equation*}
    \E \left\{ \frac{n_1 \ldots n_k}{r_k} \left\| U_{1}^{(i_1)} U_{2}^{(i_2)} \ldots U_{k}^{(i_k)} \right\|_2^2\right\} = r_1 \ldots r_{k-1} \leq r^{k - 1}.
\end{equation*}
The exponential dependence on $k$ leads to the exponential dependence on $d$ in the sample complexity via Lemma \ref{guarantee:lemma:core_projection_bound}. 

\begin{figure}
    \centering
    \includegraphics[width=0.45\linewidth]{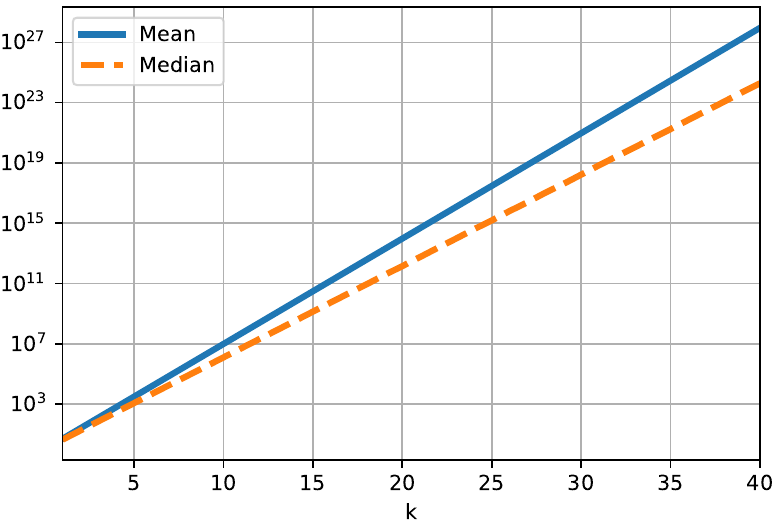}
    \caption{Numerically computed median of $\prod_{j = 1}^k \chi^2(5)$.}
    \label{discussion:fig:chi_median}
\end{figure}

It is possible, however, that the distribution is not concentrated around the expected value but is spread out, i.e. the majority of random row-vectors $U_{1}^{(i_1)} U_{2}^{(i_2)} \ldots U_{k}^{(i_k)}$ has very small norms. In other words, for a significant subset of multi-indices $\omega$ the projections $\| \mathcal{P}_{\bm{A}} \bm{E}_\omega \|_F$ might be small. In this case, the Bernstein inequality, which is the crux of Theorem \ref{guarantee:theorem:rip_core}, can produce crude estimates---as it requires a uniform upper bound of the random variable that holds almost surely---and a different tail bound such as in \cite{Maurerbound2003} could lead to finer results. To check this hypothesis, we can estimate the \textit{median} of $s_1(r_1) \ldots s_k(r_{k})$ in a numerical simulation. Unfortunately, the results in Figure~\ref{discussion:fig:chi_median} show that the median grows exponentially too, and so the squared norm is of order $r^k$ for many row-vectors.

\begin{figure}
\centering
	\includegraphics[width=0.45\textwidth]{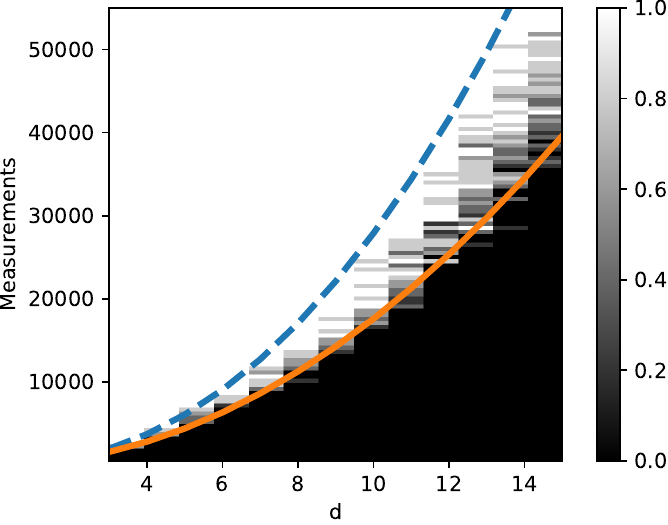}
\caption{Phase plot of the Riemannian gradient descent for $n = 50$, $r = 3$, and varying number of dimensions $d$. The values between 0 and 1 are the frequencies of successful recovery for the given parameters. The orange (solid) and blue (dashed) curves correspond to $|\Omega| = d^2 r^2 n \log(n) / 10$ and $|\Omega| = d^{2.2} r^2 n \log(n) / 10$, respectively.}
\label{discussion:fig:dims}
\end{figure}

Still, we hope that the `true` estimate of $|\Omega|$ should not depend exponentially on the number of dimensions $d$, the phase plot in Figure~\ref{discussion:fig:dims} supports our hopes. We applied the RGD \eqref{complt:eq:rgd} to TT completion with $n = 50$, $\bm{r} = (3, \ldots, 3)$, and varying number of dimensions $d$ and sample size $|\Omega|$; for every combination of $d$ and $|\Omega|$ we carried out 5 random experiments. In each of them, we 1) generated a random tensor $\bm{A}$ and a random initial approximation $\bm{X}_0$, both of TT rank $\bm{r}$, with i.i.d. standard Gaussian TT cores; 2) generated a uniformly distributed sampling set $\Omega_1$ and a uniformly distributed test set $\Omega_2$, both of size $|\Omega|$; 3) ran 500 iterations of the RGD with data $\mathcal{R}_{\Omega_1} \bm{A}$ starting from $\bm{X}_0$; 4) and called the iterations successful if the relative error on the test set $\Omega_2$ was below $10^{-4}$:
\begin{equation*}
    \| \mathcal{R}_{\Omega_2} \bm{A} - \mathcal{R}_{\Omega_2} \bm{X}_{500}\|_F < 10^{-4} \| \mathcal{R}_{\Omega_2} \bm{A} \|_F.
\end{equation*}
The implementation of the RGD was taken from the TTeMPS Toolbox (the RTTC method, see \url{https://www.epfl.ch/labs/anchp/index-html/software/ttemps/}). The phase plot in Figure~\ref{discussion:fig:dims} shows the frequency of success for every combination of $d$ and $|\Omega|$. We see that the phase transition curve between the `never successful` (black) and `always successful` (white) regions seems to exhibit polynomial growth. 

The practical implications of the polynomial dependence are best seen on the following example. Consider a dataset of $2^d$ real numbers, of which we know only $|\Omega|$. We can rearrange it into a $d$-dimensional tensor and, if $r$ is its largest TT rank, recover all of the data from $|\Omega| \gtrsim d^3 r^2$ elements, as Figure~\ref{discussion:fig:dims} suggests. If, instead, we ignore the tensor structure and treat the data as a $2^{d-1} \times 2^{d-1}$ matrix, the same bound evaluates to $|\Omega| \gtrsim d 2^d r^2$. Therefore, using the tensor structure (if it exists) of a large dataset, we shall manage to recover it from significantly fewer elements.

A different kind of reasoning might be needed to bridge the gap between the theoretical exponential bound \eqref{guarantee:theorem:rip_core} and the numerical polynomial bound from Figure~\ref{discussion:fig:dims}. One possible direction can lie in relaxing the RIP \eqref{complt:eq:rip}. Currently, it states that the sampling operator $\mathcal{R}_{\Omega}$ is well-conditioned on the \textit{whole} tangent space $T_{\bm{A}} \mathcal{M}_{\bm{r}}$. Recent research into non-convex optimization for matrix completion and phase retrieval, however, shows that the gradients are far from being arbitrary and enjoy good entrywise bounds \cite{ChenEtAlGradient2019, DingChenLeave2020} when the initial condition is incoherent with respect to the measurement operator. So, by adapting the RIP \eqref{complt:eq:rip} to gradients with entrywise bounds, it might be possible to reduce the sample complexity in Theorem~\ref{guarantee:theorem:rip_core}.

\section*{Acknowledgements}
We are grateful to the referees, whose thoughtful comments helped us present our ideas in a clearer and more structured way. This work was supported by Russian Science Foundation (project 21-71-10072).

\appendix
\section{Other approaches to matrix and tensor completion}\label{section:related}
Given the success of nuclear norm minimization for matrices---in terms of both computational feasibility and sample complexity---the transition to the multi-dimensional case did not suffer from the lack of ideas. The nuclear norm heuristic was extended as a convex surrogate of Tucker (also known as multilinear) ranks \cite{GandyEtAlTensor2011, SignorettoEtAlNuclear2010, LiuEtAlTensor2013} and TT ranks \cite{BenguaEtAlEfficient2017} by setting the cost function to the sum of the nuclear norms (SNN) of the tensor flattenings or unfoldings.

The Tucker/multilinear ranks of a $d$-dimensional tensor $\bm{A}$ are defined as a tuple of ranks of all the mode-$j$ flattenings
\begin{equation*}
    \text{rank}_\text{Tucker} (\bm{A}) = (\text{rank}(A_{(1)}), \ldots, \text{rank}(A_{(d)})).
\end{equation*}
Assume for simplicity that all the sizes are equal to $n$ and all the Tucker/multilinear ranks are equal to $r$. The sample complexity of SNN for Tucker recovery from random Gaussian measurements was studied in \cite{TomiokaEtAlEstimation2011, MuEtAlSquare2014}. Tucker completion via SNN was treated in \cite{huang2015provable} where the authors assumed the incoherence \eqref{intro:eq:coherence} of one of the mode-$j$ flattenings $A_{(j)}$; the RIP \eqref{intro:eq:rip} was proved to hold with high probability if the sample $\Omega \subseteq [n]^d$ contains
\begin{equation*}
    |\Omega| \gtrsim \mu_0 d r n^{d-1} \log(n)
\end{equation*}
randomly chosen elements. With the help of an additional mutual incoherence property of the tensor, it was proved that SNN can recover $\bm{A}$ with high probability if
\begin{equation*}
    |\Omega| \gtrsim \mu_0 d^4 r n^{d-1} \log^2(n).
\end{equation*}

A different view on the tensor nuclear norm and tensor completion consists in extending the spectral norm and taking its dual \cite{YuanZhangTensor2016}. This approach, however, is mostly of theoretical value: the norm in question is computationally intractable but leads to improved estimates of the sample size compared to SNN. In \cite{YuanZhangIncoherent2017}, a special incoherent nuclear norm was constructed for the Tucker completion problem. An analog of the RIP \eqref{complt:eq:rip} was proved to hold with high probability under the incoherence assumption \eqref{intro:eq:incoherent} for all mode-$j$ fiber spans provided that
\begin{equation*}
    |\Omega| \gtrsim \mu_0^{d-1} d r^{d-1} n \log(n)
\end{equation*}
samples are drawn uniformly at random. The minimization of the incoherent nuclear norm was proved to recover $\bm{A}$ when
\begin{equation*}
    |\Omega| \gtrsim C_d  (\mu_0^{d-1} r^{d-1} n + \mu_0^{\frac{d-1}{2}} r^{\frac{d-1}{2}} n^{\frac{3}{2}}) \log^2(n), \quad C_d = C_d(d).
\end{equation*}

Another approach to tensor completion is to minimize the residual under the rank constraint (as in Eq.~\eqref{intro:eq:nonconvex}), without going into the geometric nuances on Riemannian optimization. In the matrix case, the singular value projection (SVP) algorithm \cite{JainEtAlGuaranteed2010} (see also a closely related iterative hard thresholding algorithm \cite{TannerWeiNormalized2013}) was developed as a projected gradient descent method
\begin{equation}
\label{intro:eq:svp}
    X_{t + 1} = \text{SVD}_r \left(X_t - \frac{\rho^{-1}}{1 + \delta_{2r}} \left[ \mathcal{R}_{\Omega} X_t - \mathcal{R}_{\Omega} A \right]  \right), \quad X_0 = 0.
\end{equation}
Here, $0 < \delta_{2r} < 1$ is a RIP constant, where RIP is understood as
\begin{equation*}
    (1 - \delta_{2r}) \| X \|_F^2 \leq \rho^{-1} \| \mathcal{R}_{\Omega}X \|_F^2 \leq (1 + \delta_{2r}) \| X \|_F^2
\end{equation*}
for all matrices of rank at most $2r$ \textit{and} with bounded coherence \eqref{intro:eq:incoherent}. This RIP holds with high probability when
\begin{equation*}
    |\Omega| \gtrsim \mu_0^2 r^2 n \log(n),
\end{equation*}
which exceeds what is required by Theorem 1 for the RIP \eqref{intro:eq:rip}. The convergence of the SVP is, however, only conjectured in Ref.~\cite{JainEtAlGuaranteed2010}: the problem is that $X_{t+1} - X_t$ and $X_t - A$ need to have uniformly bounded coherences \eqref{intro:eq:coherence}. Linear convergence in the entrywise norm was later proved, in the symmetric case, in \cite{DingChenLeave2020} when 
\begin{equation*}
    |\Omega| \gtrsim \kappa^6 \mu_0^4 r^6 n \log(n),
\end{equation*}
where $\kappa$ is the condition number (in the spectral norm) of $A$.

The SVP framework has been extended to tensor recovery in Tucker and TT formats \cite{RauhutEtAlTensor2015, RauhutEtAlLow2017} under the assumption that the measurement operator satisfies the standard RIP \eqref{recov:eq:rip_vanilla}. In the multi-dimensional setting, the exact SVD-based matrix projection is replaced with HOSVD \cite{DeLathauwer2000multilinear} and TT-SVD \cite{OseledetsTensorTrain2011}, which are the standard generalizations of SVD to the Tucker and TT formats. The main difference between the matrix and tensor cases is that the truncated HOSVD and TT-SVD are quasi-optimal projections as opposed to the optimal truncated SVD. The theory of the SVP convergence for matrices has been extended to quasi-optimal projections \cite{LebedevaEtAlLowRank2021}. For HOSVD and TT-SVD, the quasi-optimality constant is rather large, $\sqrt{d}$, a fact that poses problems for theoretical analysis (but less so for practical purposes since $\sqrt{d}$ corresponds to the worst case). That is why a local optimality assumption accompanies the standard RIP of order $3\bm{r}$---note that matrix SVP requires the standard RIP of order $2r$ (see \cite{JainEtAlGuaranteed2010})---in the proof of global SVP convergence for tensor recovery \cite{RauhutEtAlTensor2015, RauhutEtAlLow2017}. We are not aware of any theoretical results about tensor completion using SVP.

An iteration of Riemannian gradient descent for Tucker recovery can be written with the help of notation we introduced above:
\begin{equation*}
    \bm{X}_{t+1} = \text{HOSVD}_{\bm{r}} \left(\bm{X}_t - \alpha_t \mathcal{P}_{T_{\bm{X}_t}} [\mathcal{R} \bm{X}_t - \mathcal{R} \bm{A}] \right).
\end{equation*}
Its local convergence was proved in \cite{RauhutEtAlTensor2015} for $\mathcal{R}$ satisfying RIP of order $3\bm{r}$, which was improved to $2\bm{r}$ in \cite{CaiEtAlProvable2021}. The authors of the latter also show that one step of Tucker-SVP with zero initial condition gives an estimate that is sufficiently close to $\bm{A}$ for local convergence to start working. Riemannian Tucker and TT completion were studied in \cite{KressnerEtAlLowrank2014, SteinlechnerRiemannian2016} but the number of samples was estimated only numerically.

The RGD for Tucker recovery was proposed in \cite{RauhutEtAlTensor2015} and proved to converge locally for the measurement operators satisfying the standard RIP \eqref{recov:eq:rip_vanilla} of order $3\bm{r}$, which was improved to $2\bm{r}$ in \cite{CaiEtAlProvable2021}. The authors of the latter also show that one step of Tucker-SVP with zero initial condition gives an estimate that is sufficiently close to the solution for the local convergence to start working. The algorithms for Riemannian Tucker and TT completion were studied in \cite{KressnerEtAlLowrank2014, SteinlechnerRiemannian2016} but the number of samples was estimated only numerically.

By comparing the current state of affairs in matrix and tensor completion, we can now see what principal difficulties are brought in by the multiple dimensions. For matrices, the nuclear norm formulation appeared to be a perfect object from the theoretical point of view. Indeed, it exhibits both polynomial computational complexity and nearly optimal sample complexity. Meanwhile, the computable SNN model leads to poor recovery guarantees for Tucker completion, and the tightest known sample complexity is achieved by the computationally intractable incoherent nuclear norm. Likewise, if we look at the development of SVP and Riemannian optimization for matrix and tensor completion in parallel, we will note that the RIP of the sampling operator and the recovery guarantees for tensor completion are only beginning to be explored in the literature.
\section{Tensor train completion with auxiliary subspace information}\label{section:side}
Typically, algorithms for matrix and tensor completion with subspace information  are developed as generalizations of the methods used in usual matrix/tensor completion. The nuclear norm minimization approach was used in \cite{XuEtAlSpeedup2013, JainDhillonProvable2013, LedentEtAlFine2021} to recover a low-rank matrix $A$ with additional subspace information \eqref{side:eq:subspaces}:
\begin{equation}
\label{intro:eq:nuclear_side}
    \| W \|_* \to \min \quad \text{s.t.} \quad \mathcal{R}_{\Omega} (Q_1 W Q_2^T) = \mathcal{R}_{\Omega}A.
\end{equation}
It is claimed in \cite{XuEtAlSpeedup2013} that $Q_1^T A Q_2 \in \Real^{m_1 \times m_2}$ is the unique solution to \eqref{intro:eq:nuclear_side} when
\begin{equation*}
    |\Omega| \gtrsim \mu^2 r m \log(m) \log(n), \quad n = \max(n_1, n_2), \quad m = \max(m_1, m_2),
\end{equation*}
indices are chosen uniformly at random. Note that the estimate depends only logarithmically on the matrix size $n$ (cf. Theorem~\ref{intro:theorem:nnm_recovery}). The coefficient $\mu^2$ depends on the coherences \eqref{intro:eq:coherence} of the column and row spaces of $A$ and of the additionally known subspaces spanned by $Q_1$ and $Q_2$. 

In \cite{BudzinskiyZamarashkinNote2020}, a Riemannian algorithm for TT completion with subspace information was proposed and its sample complexity was studied numerically. Here, we want to address the question from the theoretical point of view. Tensor completion with subspace information for other tensor formats was treated in \cite{BudzinskiyZamarashkinVariational2022, LongEtAlTrainable2022}. 

\subsection{Riemannian gradient descent}
First of all, we need to establish the geometry of the problem. Denote by $\mathcal{M}_{\bm{r}}^{(m)}$ and $\mathcal{M}_{\bm{r}}^{(n)}$ the submanifolds of small $m_1 \times \ldots \times m_d$ and large $n_1 \times \ldots \times n_d$ tensors of TT rank $\bm{r}$, respectively. The following Lemma~\ref{side:lemma:product_rank} shows that the size of the tensor can be increased by applying mode-$k$ products without altering the rank.

\begin{lemma}\label{side:lemma:product_rank}
Let $S_k \in \Real^{n_k \times m_k}$ be a matrix of rank $m_k$. Then for any tensor $\bm{W} \in \Real^{m_1 \times \ldots \times m_d}$ the mode-$k$ product with $S_k$ does not change its TT rank:
\begin{equation*}
    \mathrm{rank}_{TT}(\bm{W}) = \mathrm{rank}_{TT}(\bm{W} \times_k S_k).    
\end{equation*}
\end{lemma}
\begin{proof}
Let $\bm{W} = [\bm{C}_1, \ldots, \bm{C}_d]$ be a minimal TT representation of $\bm{W} \in \Real^{m_1 \times \ldots \times m_d}$. By definition of the mode-$k$ product,
\begin{equation*}
    \bm{W} \times_k S_k = [\bm{C}_1, \ldots, \bm{C}_{k-1}, \bm{D}_k, \bm{C}_{k+1}, \bm{C}_d], \quad \bm{D}_k = \bm{C}_k \times_2 S_k \in \Real^{r_{k-1} \times n_k \times r_k}.
\end{equation*}
It suffices to show that this TT representation is also minimal, i.e. that the left and right unfoldings of $\bm{D}_k$ are full-rank. It is easy to see that the left and right unfoldings  satisfy
\begin{equation*}
    D_k^L = (S_k \otimes I_{r_{k-1}})C_k^L, \quad D_k^R = C_k^R (S_k^T \otimes I_{r_{k}})
\end{equation*}
and are full-rank as products of full-rank matrices.
\end{proof}

Thanks to Lemma \ref{side:lemma:product_rank}, the linear operator $\mathcal{Q}: \Real^{m_1 \times \ldots \times m_d} \to \Real^{n_1 \times \ldots \times n_d}$ defined by
\begin{equation*}
    \mathcal{Q} \bm{W} = \bm{W} \times_1 Q_1 \times_2 \ldots \times_d Q_d,
\end{equation*}
where $Q_k$ are matrices with orthonormal columns, can be restricted to the submanifold $\mathcal{M}^{(m)}_{\bm{r}}$ as $\mathcal{Q}: \mathcal{M}^{(m)}_{\bm{r}} \to \mathcal{M}^{(n)}_{\bm{r}}$. Its image $\mathcal{Q}(\mathcal{M}^{(m)}_{\bm{r}})$ is an embedded submanifold \cite{LeeIntroduction2012} of $\mathcal{M}^{(n)}_{\bm{r}}$ and the adjoint operator
\begin{equation*}
    \mathcal{Q}^* \bm{X} = \bm{X} \times_1 Q_1^T \times_2 \ldots \times_d Q_d^T
\end{equation*}
acts as the left inverse $\mathcal{Q}^* \mathcal{Q} = \mathrm{Id}$. The set $\mathcal{Q}(\mathcal{M}^{(m)}_{\bm{r}})$ contains precisely those tensors that satisfy the rank and subspace requirements (this explains Eq.~\eqref{side:eq:subspaces}).

\begin{lemma}
\label{side:lemma:manifold}
Let $\bm{A} \in \mathcal{M}^{(n)}_{\bm{r}}$ be a tensor of TT rank $\bm{r}$. All of its mode-$k$ fiber spans belong to the subspaces spanned by the columns of $Q_k$ if and only if $\bm{A} \in \mathcal{Q}(\mathcal{M}^{(m)}_{\bm{r}})$.
\end{lemma}
\begin{proof}
Let $\bm{A} = \mathcal{Q} \bm{B}$ with $\bm{B} \in \mathcal{M}^{(m)}_{\bm{r}}$. Then by the definition of the mode-$k$ product
\begin{equation*}
    A_{(k)} = Q_k B_{(k)}, \quad k \in [d],
\end{equation*}
and the inclusion of subpsaces follows. Conversely, let the mode-$k$ fiber spans of $\bm{A}$ belong to the column spans of $Q_k$. Then $Q_k Q_k^T A_{(k)} = A_{(k)}$ and $\mathcal{Q} \mathcal{Q}^* \bm{A} = \bm{A}$. Then $\bm{B} = \mathcal{Q}^* \bm{A}$ lies in $\mathcal{M}^{(m)}_{\bm{r}}$ since if it had different TT ranks, so would $\mathcal{Q} \bm{B}$ by Lemma \ref{side:lemma:product_rank}.
\end{proof}

This means that Riemannian optimization can be applied to TT completion with subspace information, and we only need to narrow down the manifold:
\begin{equation*}\label{side:eq:optim_narrow}
    \| \sqrt{\mathcal{R}}_{\Omega} \bm{X} - \sqrt{\mathcal{R}}_{\Omega} \bm{A} \|_F^2 \to \min \quad \text{s.t.} \quad \bm{X} \in \mathcal{Q}(\mathcal{M}^{(m)}_{\bm{r}}).
\end{equation*}
The projection onto the new tangent space can be easily computed as
\begin{equation*}
    \mathcal{P}_{T_{\bm{X}} \mathcal{Q}(\mathcal{M}^{(m)}_{\bm{r}})} = \mathcal{Q}^* \mathcal{P}_{T_{\bm{X}} \mathcal{M}^{(n)}_{\bm{r}}}
\end{equation*}
and the formulation of the RGD follows immediately. However, for the theoretical analysis we prefer to use an equivalent optimization problem that works on $\mathcal{M}^{(m)}_{\bm{r}}$ rather than directly on $\mathcal{Q}(\mathcal{M}^{(m)}_{\bm{r}})$. Let $\bm{A} = \mathcal{Q} \bm{B}$ with $\bm{B} \in \mathcal{M}^{(m)}_{\bm{r}}$, then we consider
\begin{equation*}
    \| \sqrt{\mathcal{R}}_{\Omega} \mathcal{Q} \bm{W} - \sqrt{\mathcal{R}}_{\Omega} \mathcal{Q} \bm{B} \|_F^2 \to \min \quad \text{s.t.} \quad \bm{W} \in \mathcal{M}^{(m)}_{\bm{r}}.
\end{equation*}
The modified sampling operator is $\mathcal{Q}^* \mathcal{R}_{\Omega} \mathcal{Q}$ and a step of the RGD can be written as
\begin{equation*}
\label{side:eq:rgd}
    \bm{W}_{t+1} = \text{TT-SVD}_{\bm{r}} \left( \bm{W}_t - \alpha_t \bm{Y}_t \right) \in \mathcal{M}^{(m)}_{\bm{r}}, \quad
    \bm{Y}_t = \mathcal{P}_{\bm{W}_t} [\mathcal{Q}^* \mathcal{R}_{\Omega} \mathcal{Q} \bm{W}_t - \mathcal{Q}^* \mathcal{R}_{\Omega} \mathcal{Q} \bm{B}] \in T_{\bm{W}_t} \mathcal{M}^{(m)}_{\bm{r}}
\end{equation*}
with the step size
\begin{equation*}
\label{side:eq:step_size}
    \alpha_t = \frac{\| \bm{Y}_t \|_F^2}{\langle \mathcal{Q}^* \mathcal{R}_{\Omega} \mathcal{Q} \bm{Y}_t, \bm{Y}_t \rangle_F}.
\end{equation*}

The RIP \eqref{complt:eq:rip}, Lemma \ref{complt:lemma:rip_rip}, and Theorem \ref{complt:theorem:convergence} for TT completion undergo simple modifications according to 
\begin{equation}
\label{side:eq:rip}
    \Big\| \mathcal{P}_{T_{\bm{B}} \mathcal{M}^{(m)}_{\bm{r}}} - \rho^{-1} \mathcal{P}_{T_{\bm{B}} \mathcal{M}^{(m)}_{\bm{r}}} \mathcal{Q}^* \mathcal{R}_{\Omega} \mathcal{Q} \mathcal{P}_{T_{\bm{B}} \mathcal{M}^{(m)}_{\bm{r}}} \Big\| < \varepsilon, \quad \| \mathcal{Q}^* \mathcal{R}_{\Omega} \mathcal{Q} \| \leq C.
\end{equation}
Other than this, the formulations and proofs transfer verbatim to the current scenario.
The convergence rate and the estimate of the local convergence basin that are then given in terms of $\bm{W}_t$ and $\bm{B}$ hold identically for $\mathcal{Q} \bm{W}_t$ and $\bm{A}$ since $\| \bm{W} - \bm{B} \|_F = \| \mathcal{Q} \bm{W} - \bm{A} \|_F$ and $\sigma_{\min}(\bm{B}) = \sigma_{\min}(\bm{A})$.

\subsection{Recovery guarantees}
Let us show that the these assumptions hold with high probability. First of all, note that $\| \mathcal{Q}^* \mathcal{R}_{\Omega} \mathcal{Q} \| \leq \| \mathcal{R}_{\Omega} \|$, hence Lemma \ref{guarantee:lemma:repetitions} applies. To derive probabilistic sufficient conditions for the modified RIP \eqref{side:eq:rip}, we need to prove analogs of Lemma~\ref{guarantee:lemma:core2interface}, Lemma~\ref{guarantee:lemma:core_projection_bound}, and Theorem~\ref{guarantee:theorem:rip_core}.

\begin{lemma}
\label{side:lemma:interface_projection_bound}
Let $\bm{A} = \mathcal{Q} \bm{B} \in \mathcal{M}^{(n)}_{\bm{r}}$ be a tensor of TT rank $\bm{r}$ with bounded core coherence $\mu_C(\bm{A}) \leq \mu_1$. Then for every $k \in [d-1]$ and for all multi-indices $(i_1, \ldots, i_d) \in [n_1] \times \ldots \times [n_d]$ we have
\begin{align*}
    \frac{n_1 \ldots n_k}{r_k} \| P_{\leq k} (Q_k^T e_{i_k} \otimes \ldots \otimes Q_1^T e_{i_1}) \|_2^2 \leq \mu_1^k, \quad 
    \frac{n_{k+1} \ldots n_d}{r_k} \| P_{\geq k+1} (Q_{k+1}^T e_{i_{k+1}} \otimes \ldots \otimes Q_d^T e_{i_d}) \|_2^2 \leq \mu_1^{d - k},
\end{align*}
where $P_{\leq k}$ and $P_{\geq k+1}$ are the orthogonal projections onto the column spans of the interface matrices $B_{\leq k}$ and $B_{\geq k+1}$.
\end{lemma}
\begin{proof}
Let $\bm{B} = [\bm{U}_1, \ldots, \bm{U}_{d-1}, \bm{G}_d]$ be a minimal left-orthogonal TT representation. Then $\bm{S}_k \in \Real^{r_{k-1} \times n_k \times r_k}$ defined as
\begin{equation*}
    S_k^L = (Q_k \otimes I_{r_{k-1}}) U_k^L
\end{equation*}
give a minimal left-orthogonal TT representation of $\bm{A}$. Denote by $\xi_k$ the $r_{k}$-dimensional row-vector whose norm we need to estimate
\begin{equation*}
    \xi_k = U_{\leq k}^T (Q_k^T e_{i_k} \otimes \ldots \otimes Q_1^T e_{i_1}).
\end{equation*}
Given the recursive formula \eqref{eq:tt:proj_recursive} we establish that
\begin{equation*}
    \xi_k = (U_k^L)^T [Q_k^T e_{i_k} \otimes \xi_{k-1}], \quad \xi_0 = 1.
\end{equation*}
The incoherence assumption for $\bm{A}$ tells us that
\begin{equation*}
    \max_{i \in [n_k]} \| S_{k}^{(i)} \|_2^2 \leq \frac{r_{k}}{r_{k-1} n_k} \mu_1
\end{equation*}
and so since
\begin{equation*}
    S_{k}^{(i_k)} = (e_{i_k}^T \otimes I_{r_{k-1}}) S_k^L = (e_{i_k}^T Q_k \otimes I_{r_{k-1}}) U_k^L,
\end{equation*}
we obtain
\begin{equation*}
    \| \xi_k \|_2^2 = \| (U_k^L)^T [Q_k^T e_{i_k} \otimes \xi_{k-1}] \|_2^2 = \| (S_{k}^{(i_k)})^T \xi_{k-1} \|_2^2 \leq \frac{r_{k}}{r_{k-1} n_k} \mu_1 \| \xi_{k-1} \|_2^2 \leq \frac{r_k}{n_1 \ldots n_k} \mu_1^k.
\end{equation*}
The argument is the same for the right unfoldings.
\end{proof}

\begin{lemma}
\label{side:lemma:side_projection_bound}
Let $\bm{A} = \mathcal{Q} \bm{B} \in \mathcal{M}^{(n)}_{\bm{r}}$ be a tensor of TT rank $\bm{r}$ with bounded core coherence $\mu_C(\bm{A}) \leq \mu_1$. Assume that the coherences of the auxiliary subspaces are bounded as well $\mu(Q_k) \leq \mu_2$. Then for every canonical basis tensor $\bm{E}_{\omega} \in \Real^{n_1 \times \ldots \times n_d}$, $\omega \in [n_1] \times \ldots \times [n_d]$, its projection onto the tangent space $T_{\bm{B}} \mathcal{M}^{(m)}_{\bm{r}}$ can be bounded from above as
\begin{equation*}
    \Big\| \mathcal{P}_{T_{\bm{B}} \mathcal{M}^{(m)}_{\bm{r}}} \mathcal{Q}^* \bm{E}_\omega \Big\|_F^2 \leq \frac{\mu_1^{d-1} \mu_2}{n_1 \ldots n_d} \sum_{k = 1}^d r_{k-1} m_k r_k.
\end{equation*}
\end{lemma}
\begin{proof}
We have
\begin{equation*}
    \Big\| \mathcal{P}_{T_{\bm{B}} \mathcal{M}^{(m)}_{\bm{r}}} \mathcal{Q}^* \bm{E}_\omega \Big\|_F^2 \leq \| \mathcal{P}_{\geq 2} \mathcal{Q}^* \bm{E}_\omega \|_F^2 + \sum_{k = 2}^{d-1} \| \mathcal{P}_{\leq k-1} \mathcal{P}_{\geq k+1} \mathcal{Q}^* \bm{E}_\omega \|_F^2 + \| \mathcal{P}_{\leq d-1} \mathcal{Q}^* \bm{E}_\omega \|_F^2.
\end{equation*}
For the first and last terms we obtain
\begin{equation*}
    \| \mathcal{P}_{\geq 2} \mathcal{Q}^* \bm{E}_\omega \|_F^2 = \| Q_1^T e_{i_1} \circ P_{\geq 2} (Q_2^T e_{i_{2}} \otimes \ldots \otimes Q_d^T e_{i_d}) \|_F^2 \leq \frac{m_1}{n_1} \mu_2 \frac{r_1}{n_2 \ldots n_d} \mu_1^{d-1}
\end{equation*}
and
\begin{equation*}
    \| \mathcal{P}_{\leq d-1} \mathcal{Q}^* \bm{E}_\omega \|_F^2 = \| P_{\leq d-1} (Q_{d-1}^T e_{i_{d-1}} \otimes \ldots \otimes Q_1^T e_{i_1})~\circ~Q_d^T e_{i_d} \|_F^2 \leq \frac{r_{d-1}}{n_1 \ldots n_{d-1}} \mu_1^{d-1} \frac{m_d}{n_d} \mu_2.
\end{equation*}
The summands $\| \mathcal{P}_{\leq k-1} \mathcal{P}_{\geq k+1} \mathcal{Q}^* \bm{E}_\omega \|_F^2$ in the middle are equal to
\begin{align*}
    \| P_{\leq k-1}(Q_{k-1}^T e_{i_{k - 1}} \otimes \ldots \otimes Q_1^T e_{i_1}) &\circ~Q_k^T e_{i_k} \circ P_{\geq k + 1} (Q_{k+1}^T e_{i_{k+1}} \otimes \ldots \otimes Q_d^T e_{i_d}) \|_F^2 \\
    &\leq \frac{r_{k-1}}{n_{1} \ldots n_{k-1}} \mu_1^{k - 1} \frac{m_k}{n_k} \mu_2 \frac{r_{k}}{n_{k+1} \ldots n_d} \mu_1^{d - k}.
\end{align*}
It remains to combine the estimates.
\end{proof}

\begin{theorem}\label{side:theorem:rip}
Let $\bm{A} = \mathcal{Q} \bm{B} \in \mathcal{M}^{(n)}_{\bm{r}}$ be a tensor of TT rank $\bm{r}$ with bounded core coherence $\mu_C(\bm{A}) \leq \mu_1$. Assume that the coherences of the auxiliary subspaces are bounded as well $\mu(Q_k) \leq \mu_2$ and let $\Omega \subset [n_1] \times \ldots \times [n_d]$ be a collection of indices sampled uniformly at random with replacement. Then the modified RIP \eqref{side:eq:rip}
\begin{equation*}
    \Big\| \mathcal{P}_{T_{\bm{B}} \mathcal{M}^{(m)}_{\bm{r}}} - \rho^{-1} \mathcal{P}_{T_{\bm{B}} \mathcal{M}^{(m)}_{\bm{r}}} \mathcal{Q}^* \mathcal{R}_{\Omega} \mathcal{Q} \mathcal{P}_{T_{\bm{B}} \mathcal{M}^{(m)}_{\bm{r}}} \Big\| < \varepsilon, \quad \rho = \frac{|\Omega|}{n_1 \ldots n_d},
\end{equation*}
holds with probability at least $1 - 2m^{d(1 - \beta)}$, $m = \max( m_1, \ldots, m_d )$, for all $\beta > 1$ provided that
\begin{equation*}
    |\Omega| \geq \frac{8}{3} \frac{\beta}{\varepsilon^2} \mu_1^{d-1} \mu_2 \left( \sum_{k = 1}^d r_{k-1} m_k r_k \right) d\log(m).
\end{equation*}
\end{theorem}
\begin{proof}
For an arbitrary tensor $\bm{Z} \in \Real^{m_1 \times \ldots \times m_d}$ we can represent $\mathcal{Q} \bm{Z}$ as 
\begin{equation*}
    \mathcal{Q} \bm{Z} = \sum_{\omega \in [n_1] \times \ldots \times [n_d]} \langle \mathcal{Q} \bm{Z}, \bm{E}_{\omega} \rangle_F \bm{E}_{\omega}.
\end{equation*}
Denote by $\mathcal{P}_{\bm{B}}$ the projection $\mathcal{P}_{T_{\bm{B}} \mathcal{M}^{(m)}_{\bm{r}}}$. It follows that
\begin{equation*}
    \mathcal{P}_{\bm{B}} \bm{Z} = \sum_{\omega \in [n_1] \times \ldots \times [n_d]} \langle \bm{Z}, \mathcal{P}_{\bm{B}} \mathcal{Q}^* \bm{E}_{\omega} \rangle_F \mathcal{P}_{\bm{B}} \mathcal{Q}^* \bm{E}_{\omega}
\end{equation*}
and
\begin{equation*}
    \mathcal{P}_{\bm{B}} \mathcal{Q}^* \mathcal{R}_\Omega \mathcal{Q} \mathcal{P}_{\bm{B}} \bm{Z} = \sum_{\omega \in \Omega} \langle \bm{Z}, \mathcal{P}_{\bm{B}} \mathcal{Q}^* \bm{E}_\omega \rangle_F \mathcal{P}_{\bm{B}} \mathcal{Q}^* \bm{E}_\omega.
\end{equation*}
As we introduce operators $\mathcal{S}_\omega : \Real^{m_1 \times \ldots \times m_d} \to \Real^{m_1 \times \ldots \times m_d}$ defined by
\begin{equation*}
    \mathcal{S}_\omega \bm{Z} = \langle \bm{Z}, \mathcal{P}_{\bm{B}} \mathcal{Q}^* \bm{E}_\omega \rangle_F \mathcal{P}_{\bm{B}} \mathcal{Q}^* \bm{E}_\omega
\end{equation*}
the proof follows the proof of Theorem \ref{guarantee:theorem:rip_interface}.
\end{proof}

We believe that Theorem~\ref{side:theorem:rip} is the first theoretical estimate for the sample complexity of TT completion with subspace information. Previous results on matrix completion with subspace information contained a $\log(n)$ factor in the sample complexity \cite{XuEtAlSpeedup2013}; our bound, which guarantees the local convergence of the RGD, depends only on the dimensions of the auxiliary subspaces and not on the dimensions of the tensor:
\begin{equation*}
    |\Omega| \gtrsim \mu_1^{d-1} \mu_2 d^2 r^2 m \log(m).
\end{equation*}
This behavior is further well-aligned with the numerical experiments carried out in \cite{BudzinskiyZamarashkinNote2020}.

\printbibliography

\end{document}